\documentclass[10pt]{article}

\usepackage{latexsym,color,amsmath,amsthm,amssymb,amscd,amsfonts}

\newtheorem{theo}{Theorem}

\newtheorem{cor}[theo]{Corollary}
\newtheorem{prop}[theo]{Proposition}
\newtheorem{defn}[theo]{Definition}

\def\R{\R}

\def\R{{\mathbb R}}
\def\grad{\nabla}

\def\qed{\hfill $\vcenter{\hrule height .3mm
\hbox {\vrule width .3mm height 2.1mm \kern 2mm \vrule width .3mm
height 2.1mm} \hrule height .3mm}$ \bigskip}

\def\lam{\lambda}

\def\pmx{\begin{pmatrix}}
\def\emx{\end{pmatrix}}

\def\det{{\rm det}}

\def\R{\mathbb R}

\begin{document}

\title{Mixed $f$-divergence and inequalities for log concave functions
}

\author{Umut Caglar and Elisabeth M.  Werner\thanks{Partially supported by an  NSF grant}}

\date{}

\maketitle
\begin{abstract}
Mixed $f$-divergences, a concept from information theory and statistics,  measure the difference between multiple pairs of distributions. 
We introduce them for log concave  functions and  establish some of their  properties.  Among them  are affine invariant vector entropy inequalities, like 
 new  Alexandrov-Fenchel type inequalities and an affine isoperimetric inequality for the vector form of the Kullback Leibler divergence for log concave functions. 
\par
Special cases of $f$-divergences are 
mixed $L_\lambda$-affine
surface areas for log concave  functions. For those, we   establish various  affine isoperimetric inequalities as well as a vector Blaschke   Santal\'{o} type inequality. 
\end{abstract}

\section{Introduction}
Affine invariant notions have had a transformative effect in convex geometry, e.g.,  \cite{GaZ, Ludwig-Reitzner,LutwakYangZhang2002, Schu, Zhang1994}. One reason for this is that there are powerful inequalities associated to those
notions. See,  for instance,   \cite{GaZ, HabSch2, Ludwig2003, Lutwak1996, LutwakYangZhang2002, LutwakYangZhang2004, WernerYe2008, WernerYe2009}. Within the last few years,  amazing connections have been discovered between some of these affine invariant notions and  concepts from information theory, e.g., \cite{Gardner2002, GLYZ, JenkinsonWerner, LutwakYangZhang2002/1, LutwakYangZhang2004/1, LutwakYangZhang2005, PaourisWerner2011}, 
leading to a totally new point of view and introducing  a whole new set of tools in the area of convex geometry. 
In particular, it was observed in \cite{Werner2012/1} that one of the most important affine invariant notions,
the $L_p$-affine surface area for convex bodies \cite{Lutwak1996, SW2004},  is R\'enyi entropy from information theory and statistics.   R\'enyi entropies are special cases of $f$-divergences.  Consequently those were then introduced for convex bodies and their corresponding entropy inequalities have been established in \cite{Werner2012}.
An $f$-divergence (see below for the precise definition) is a function that measures the difference between (probability)
densities. Aside from R\'enyi entropies, e.g.,  the relative entropy or Kullback-Leibler divergence \cite{KullbackLeibler1951}  and  the Bhattcharyya distance \cite{Bhattacharyya1946} are examples of $f$-divergences.
\par
Much effort has been devoted lately to extend concepts and inequalities that hold for convex bodies to the corresponding ones for classes of functions. A natural analogue for a convex body is a 
log concave function. For those, functional analogues of  important inequalities have been proved,  such as the Blaschke Santal\'o inequality \cite{ArtKlarMil, KBallthesis,  MeyerFradelizi2007, Lehec2009} and the affine isoperimetric inequality
\cite{ArtKlarSchuWer}. In \cite{CaglarWerner}, $f$-divergences  were introduced for log concave functions. 
This new concept yielded entropy inequalities which  are stronger than the already existing ones, the reverse log-Sobolev and the reverse Poincare inequalities  of \cite{ArtKlarSchuWer}. 
\par
Now we develop these ideas even further and  introduce  the  mixed $f$-divergence for log concave functions.
For convex bodies these were introduced and developed in \cite{WernerYe2013}.
Mixed $f$-divergence,  which is  important in  applications, such as statistical hypothesis testing and classification, see e.g., \cite{ Menendez, MoralesPardo1998, Zografos1998}, 
measures the difference between multiple pairs of (probability) distributions.
Examples  include, e.g., the Matusita's affinity \cite{Matusita1967, Matusita1971}, the Toussaint's affinity \cite{Toussaint1974}, the information radius \cite{Sibson1969} and the average divergence \cite{Sgarro1981}. 
Mixed $f$-divergence is  an extension of the classical $f$-divergence and can be viewed as a vector form of classical $f$-divergence. For a vector $\vec \varphi = (\varphi_i)_{1 \leq i \leq n}$ consisting of log concave functions   $\varphi_i: \mathbb{R}^{n}  \rightarrow [0, \infty)$ and  a vector $\vec f = (f_i)_{1 \leq i \leq n}$ consisting of  concave or convex  functions $f_i: (0, \infty) \rightarrow \mathbb{R}_+$,  we define the mixed $f$-divergence for $\vec{\varphi}= \left(\varphi_1, \cdots, \varphi_n\right)$ by
\begin{eqnarray}\label{No.1}
D_{\vec f} \left(P_{\vec \varphi} , Q_{\vec \varphi} \right) & = & {\displaystyle \int } \prod_{i=1}^{n} \left[ \varphi_i  \  f_i \left(
\frac{e^{\frac{\langle\grad\varphi_i, x \rangle}{\varphi_i}} }{\varphi_i^2} \  \mbox{det} \left[  \text{Hess} \left( -\log \varphi_i \right)\right] \right) \right]^{\frac{1}{n}} dx.
\end{eqnarray}
If all $\varphi_i$ are the same and all $f_i$  are the same, we recover the $f$-divergences of \cite{CaglarWerner}. Like those, the new  expressions are $SL(n)$ invariant. Here, 
$\nabla \varphi$ denotes the gradient and $\text{Hess}(\varphi) =  \left(\frac{\partial ^2 \varphi}{\partial x_i \partial x_j} \right)_{1 \leq i,j\leq n}$ is the Hessian 
of $\varphi$. 
\par
One of the difficulties, to introduce this notion, was to find the right expression for the densities. 
A passage from functions to convex bodies and back, lets us achieve this goal and it can be seen that the expressions (\ref{No.1}) appear naturally. This is demonstrated in
\cite{CaglarWerner}.
\par
The study of  mixed $f$-divergences   leads  us to obtain new linear, respectively,  affine invariant entropy inequalities, among them  
new Alexandrov-Fenchel type inequalities  for log concave functions.
Alexandrov-Fenchel inequality is a fundamental result in  geometry.  It is arguably one of the strongest inequality in this area as many important inequalities such as  the Brunn-Minkowski inequality and  Minkowski's first inequality
follow from Alexandrov-Fenchel inequality (see, e.g., \cite{Gardner2002, Sch}). Different generalization of Alexandrov-Fenchel inequalities for log concave functions can be found in  e.g., \cite{MilmanRotem2013}. 
Various vector entropy inequalities are consequences of this new Alexandrov-Fenchel inequality, for instance the following 
 upper bound for the vector form of the  $f$-divergence in terms of the classical $f$-divergences 
\begin{eqnarray*}
 \left[ D_{\vec f} \left( P_{\vec \varphi}, Q_{\vec \varphi} \right) \right]^n  \leq  
 \prod_{k=1}^{ n }  D_{  f_{k }}  \left( P_{\varphi_{k}},  Q_{\varphi_{k}}\right),
\end{eqnarray*}
and an affine isoperimetric inequality for the vector form of the relative entropy for normalized  log concave functions, 
\begin{eqnarray*}
D_{KL} \left(P_{\vec \varphi} , Q_{\vec \varphi} \right)  \leq 
 \log   (2 \pi)^n.
\end{eqnarray*}
We refer to Theorem \ref{Alexandrov-Fenchel} and Corollary \ref{Korrolar2} for the detailed statements and the corresponding equality characterizations.
While for the classical Alexandrov-Fenchel  inequality for convex bodies the equality characterizations are not known in general, such equality characterizations can be established for
these new Alexandrov-Fenchel inequalities  for log concave functions. To do so, we use, among other things,  the matrix version of the Brunn-Minkowski inequality
 and recently established unique solutions of certain Monge Amp\`ere differential equations \cite{Werner-Yolcu}.
\par
Mixed $L_\lambda$-affine surface areas for a vector $\vec{\varphi}$ of log-concave functions, denoted by $ as_{\lambda} (\vec{ \varphi})  $, are  special cases of  mixed $f$-divergences.  This new definition corresponds, on the level of convex bodies,  to the mixed
$L_p$-affine surface areas (see, e.g.,  \cite{Lut1987, WernerYe2009, Ye2012}),  a generalization of  $L_p$-affine surface areas.  We refer to e.g.,   \cite{HabSch2}, \cite{Ludwig-Reitzner1999}, \cite{Ludwig-Reitzner},  \cite{Lutwak1996}, \cite{MW1}, \cite{SW1990}, \cite{SW2004},\cite{ Werner2007}-\cite{ WernerYe2008} for more information on $L_p$-affine surface area for convex bodies. 
The $L_p$-affine surface areas for functions were   introduced in \cite{CFGLSW}. 
\par
We establish  several affine isoperimetric inequalities for these quantities. Among  them is a   vector Blaschke Santal\'o type inequality for log concave functions with 
barycenter at $0$, 
\begin{equation*}
 as_{\lambda} (\vec{ \varphi})  as_{\lambda} ( \vec{\varphi^\circ}) \  \leq \ (2 \pi)^n.
\end{equation*} 
Here,   $ \varphi^\circ$ is the dual function of $\varphi$, defined in (\ref{Polare}) and  $ \lambda \in [0,1]$.
\par
Please note that all the definitions and results hold,  with obvious modifications,  for $s$-concave functions as well. We refer to  \cite{CaglarWerner} for that.
\par
Throughout the paper we will assume that the convex or concave  functions $f: (0, \infty) \rightarrow \mathbb{R}$ and the 
 log concave  functions 
$ \varphi:\R^{n}\rightarrow [0, \infty)$ have enough smoothness and integrability properties so that the expressions 
considered in the statements make sense, i.e., we will always assume that $\varphi$ and $\varphi^\circ$ $\in C^2  \cap L^1(\R^n, dx)$,  
where  $C^2$  denotes the twice continuously differentiable functions, and that 
\begin{equation}\label{assume2}
 \prod_{i=1}^{n} \left[\varphi_i f_i 
\left(
\frac{e^{\frac{\langle \nabla \varphi_i, x\rangle}{\varphi_i}}
}
{\varphi_i^{2}} 
\mbox{det} \left(  \text{Hess} \left(-\ln \varphi_i \right) 
\right) \right) \right] \in   L^1(\R^n, dx).
\end{equation}

\section{ Mixed $f$-divergence} \label{Section-fdiv}
\subsection{Background on mixed $f$-divergence} \label{subsection: mixedf-div}
In information theory, probability theory and statistics, an
$f$-divergence is a function that measures the difference between two (probability)
distributions. This notion was introduced by 
Csisz\'ar \cite{Csiszar}, and independently Morimoto \cite{Morimoto1963} and Ali \& Silvery \cite{AliSilvery1966}.
\par
Let $(X, \mu)$ be a finite measure space  and let  $P=p \mu$ and  $Q=q \mu$ be  (probability) measures on $X$ that are  absolutely continuous with respect to the measure $\mu$.  
Let $f: (0, \infty) \rightarrow  \mathbb{R}$ be a convex  or a concave  function.
The $*$-adjoint function $f^*:(0, \infty) \rightarrow  \mathbb{R}$ of $f$  is defined by 
\begin{equation}\label{adjoint}
f^*(t) = t f (1/t), \ \  t\in(0, \infty).
\end{equation}
\par
It is obvious   that $(f^*)^*=f$ and that $f^*$ is again convex  if $f$ is convex,  respectively concave if $f$ is concave.
Then the $f$-divergence   $D_f(P,Q)$ of the measures $P$ and $Q$ is defined by 
\begin{equation}\label{def:fdiv2}
D_f(P,Q)=
 \int_{X} f\left(\frac{p}{q} \right) q d\mu.
\end{equation}
\vskip 2mm
It is a generalization of well known divergences, such as, the {\em variational distance},  the {\em Kullback-Leibler divergence} or {\em relative entropy}, the {\em R\'enyi divergence}  and many more.
More on $f$-divergence can be found in e.g. \cite{GarciaWilliamson2012, LieseVajda1987,  LieseVajda2006, OsterrVajda, ReidWilliamson2011, Werner2012, WernerYe2013}.
\vskip 3mm
For applications, such as  statistical hypothesis test and classification, it is important to have extension of $f$-divergence from two (probability) measures to multiple (probability) measures,  see e.g., \cite{ Menendez, MoralesPardo1998, Zografos1998}. 
\par
For $1 \leq i \leq n$, let $P_i = p_i \mu $ and $Q_i = q_i \mu $ be probability measures on $X$ that are absolutely continuous
 with respect to the measure $\mu$. We also assume that the density functions $p_i$ and $q_i$ are nonzero almost everywhere with respect to $\mu$. Denote by
$$
\vec{\textbf{P}}  =  (P_1, P_2, \cdots, P_n) , \quad \vec{\textbf{Q}}  =  (Q_1, Q_2, \cdots, Q_n) .
$$
We use $ \vec p $ and $\vec{q}$ to denote the density vectors for $ \vec{\textbf{P}} $ and $\vec{\textbf{Q}}$ respectively,
$$
\frac{ d \vec{\textbf{P}} }{d \mu}  = \vec p =  (p_1, p_2, \cdots,  p_n) , \quad \frac{ d \vec {\textbf{Q}}}{d \mu}  = \vec q =  (q_1, q_2, \cdots,  q_n) .
$$ 
For $1 \leq i \leq n$, let $f_i : (0, \infty) \rightarrow  \R_+ $  be either convex or concave functions. Denote by $\vec f $ the vector $\vec f  = (f_1, f_2, \cdots, f_n)$ and the $*$-adjoint vector of  $ \vec f$ by $\vec f^*  = (f^*_1, f^*_2, \cdots, f^*_n)$.
The mixed $f$-divergence for $(\vec f, \vec{\textbf{P}}, \vec{\textbf{Q}})$ is defined in \cite{WernerYe2013} as 
\begin{equation}\label{def:mixed-fdiv}
D_{\vec f}(\vec{\textbf{P}}, \vec{\textbf{Q}} )=
 \int_{X} \prod_{i=1}^{n} \left[f_i \left(\frac{p_i}{q_i} \right) q_i\right]^{\frac{1}{n}} d\mu.
\end{equation}
\par
If   $f_i = f$, $P_i=P$, and $Q_i =Q$, for all $1 \leq i \leq n$, then
the mixed $f$-divergence becomes the classical $f$-divergence, defined in (\ref{def:fdiv2}).
\vskip 2mm
Similarly, the mixed $f$-divergence for $(\vec f, \vec{\textbf{Q}}, \vec{\textbf{P}})$ is
\begin{equation}\label{def:mixed-fdiv2}
D_{\vec f}(\vec{\textbf{Q}}, \vec{\textbf{P}} )=
 \int_{X} \prod_{i=1}^{n} \left[f_i \left(\frac{q_i}{p_i} \right) p_i\right]^{\frac{1}{n}} d\mu.
\end{equation}
It is obvious that $D_{\vec f}(\vec{\textbf{P}}, \vec{\textbf{Q}} )= D_{\vec f^*}(\vec{\textbf{Q}}, \vec{\textbf{P}} ).$  
Therefore, it is enough to consider $D_{\vec f}(\vec{\textbf{P}}, \vec{\textbf{Q}} )$,   which we will do throughout the paper.
\vskip 2mm
We now present some examples.  For more examples and properties, see \cite{WernerYe2013}.
\vskip 2mm
\noindent
\textbf{Examples.}
\par
\noindent
1. For $1 \leq i \leq  n$, let $f_i (t) = |t -1|$. Then  the mixed $f$-divergence becomes the {\em mixed total variation}  of $\vec P$ and $\vec Q$, defined by Werner and Ye in \cite{WernerYe2013},
\begin{equation}\label{mixedtotal}
  D_{\vec f}(\vec P, \vec Q)= \int_{X}   \prod_{i=1}^{n} | p_i - q_i |^{\frac{1}{n}} d\mu.
\end{equation}
\par
\noindent
2.  For $1 \leq i \leq  n$, let $f_i (t) =  \log t$. Then  the mixed $f$-divergence is {\em mixed Kullback-Leibler divergence} or the {\em mixed relative entropy} of $\vec P$ and $\vec Q$ \cite{WernerYe2013},
\begin{equation}\label{relent}
 D_{KL}(\vec P, \vec Q)=   D_{\vec f_{+}}(\vec P, \vec Q)= \int_{X}   \prod_{i=1}^{n} \left[ q_i \log \frac{p_i}{q_i} \right]_+^{\frac{1}{n}} d\mu
\end{equation}
where for $a \in \R^n$, $a_+ = ( \max \{a_1,0\}, \max \{a_2,0\}, \cdots, \max \{a_n,0\}  )$. Recall that   {\em Kullback-Leibler divergence} or {\em relative entropy} from $P$ to $Q$ is defined as (see, e.g., \cite{CoverThomas2006})
\begin{equation}\label{relent}
 D_{KL}(P\|Q)= \int_{X} q \log \frac{p}{q} d\mu.
\end{equation}

\subsection{ Mixed $f$-divergence for log concave functions}\label{Section-mixedlogcon}

A function $\varphi: \mathbb{R}^n \rightarrow [0, \infty)$ is log concave, if it is of the form $\varphi(x) = e^{-\psi(x)}$, where $\psi: \mathbb{R}^n \rightarrow \mathbb{R}$ is a convex function.
For $1 \leq i \leq n $, we put 
\begin{equation}\label{Q,P}
q_{\varphi_i}=\varphi_i  \hskip 4mm \text{and} \hskip 4mm   p_{\varphi_i}= \varphi_i^{-1}  e^{\frac{\langle\grad\varphi_i, x \rangle}{\varphi_i}} \mbox{det} \left[  \text{Hess} \left(-\log \varphi_i \right)\right].
\end{equation}
We use the expressions (\ref{Q,P}) to define the mixed  $f$-divergences for log concave  functions. 
These quantities  are the proper ones to use in order  to define divergences for log concave functions. This  was shown in \cite{CaglarWerner}. 

\vskip 2mm
\begin{defn} \label{mixed-f-div-log}
Let $f_i: (0, \infty) \rightarrow \mathbb{R}_+$ be   convex and/or  concave functions and let $ \varphi_i:\R^{n}\rightarrow [0, \infty)$ be  log concave functions.
Then the mixed $f$-divergence for $\vec{\varphi}= \left(\varphi_1, \cdots, \varphi_n\right)$  is
\begin{eqnarray}\label{mixed-div-Logconcave1}
D_{\vec f} \left(P_{\vec \varphi} , Q_{\vec \varphi} \right) &=& 
  D_{\vec f} \left((P_{\varphi_1}, \cdots,P_{\varphi_n}) , (Q_{\varphi_1}, \cdots, Q_{\varphi_n} ) \right) = \int \prod_{i=1}^{n} \left[f_i \left(\frac{p_{\varphi_i}}{q_{\varphi_i}} \right) q_{\varphi_i} \right]^{\frac{1}{n}} dx \nonumber \\
& = & {\displaystyle \int } \prod_{i=1}^{n} \left[ \varphi_i  \  f_i \left(
\frac{e^{\frac{\langle\grad\varphi_i, x \rangle}{\varphi_i}} }{\varphi_i^2} \  \mbox{det} \left[  \text{Hess} \left( -\log \varphi_i \right)\right] \right) \right]^{\frac{1}{n}} dx.
\end{eqnarray}
\end{defn}
\vskip 2mm
\noindent
{\bf Remarks and Examples} 
(i) If we let $ \ f_i =f \ $ and $ \ \varphi_i = \varphi $, $1 \leq i \leq n$,   then we obtain the usual $f$-divergence for log concave functions,  $D_f (P_{\varphi} , Q_{\varphi})$, defined in \cite{CaglarWerner},
\begin{equation}\label{div-Logconcave1}
D_f(P_\varphi, Q_\varphi)=
 \int \varphi  \ f  \left(
\frac{e^{\frac{\langle\grad\varphi, x \rangle}{\varphi}} }{\varphi ^2} \  \mbox{det} \left[  \text{Hess} \left( -\log \varphi \right)\right] \right)dx.
\end{equation}
Thus, Definition 1 extends the definition (\ref{div-Logconcave1}) of $f$-divergence for a log concave function of \cite{CaglarWerner} and consequently the inequalities  and identities given below generalize the ones given in\cite{CaglarWerner}. This is our  motivation for Definition 1.
\par
\noindent
(ii) Similarly to (\ref{mixed-div-Logconcave1}),
\begin{eqnarray*}
&&D_{\vec f} \left(Q_{\vec \varphi} , P_{\vec \varphi} \right) =
 D_{\vec f} \left((Q_{\varphi_1}, \cdots,Q_{\varphi_n}) , (P_{\varphi_1}, \cdots, P_{\varphi_n} ) \right) = \int   \prod_{i=1}^{n} \left[f_i \left(\frac{q_{\varphi_i}}{p_{\varphi_i}} \right) p_{\varphi_i} \right]^{\frac{1}{n}} dx \nonumber \\
&&= {\displaystyle \int }  \prod_{i=1}^{n} \left[ \varphi_i^{-1} e^{\frac{\langle\grad\varphi_i, x \rangle}{\varphi_i}} \ \mbox{det} \left[ -  \text{Hess} \left(\log \varphi_i \right)\right] \ f_i  \left( \frac{  \varphi_i^2  \  e^{- \frac{\langle\grad\varphi_i, x \rangle}{\varphi_i}}    }{
 \mbox{det} \left[ \text{Hess} \left(- \log \varphi_i \right)\right]   }\right)  \right]^{\frac{1}{n}} dx.
\end{eqnarray*}
\par
\noindent
(iii) If we write a log concave function as $\varphi= e^{-\psi}$, $\psi$ convex, then (\ref{mixed-div-Logconcave1})
becomes
\begin{eqnarray}\label{mixed-div-Logconcave3}
D_{\vec f} \left(P_{\vec \varphi} , Q_{\vec \varphi} \right) & =&
 \int  \prod_{i=1}^{n} \left[ e^{-\psi_i}  \ f_i  \left(
e^{2 \psi_i - \langle\grad\psi_i, x \rangle}\  \mbox{det} \left[  \text{Hess} \psi_i \right] \right) \right]^{\frac{1}{n}} dx.
\end{eqnarray}
\par
\noindent
(iv)  For  $1 \leq i \leq n $, let $A_i$ be a $(n\times n)$  positive definite matrix, $c_i>0$ a constant and let $ \  \varphi_i(x) = c_i e^{- \frac{1}{2} \langle A_ix, x \rangle}$. Then
\begin{equation}\label{exponential1}
D_{\vec f} \left(P_{\vec \varphi} , Q_{\vec \varphi} \right) = \  \frac{( 2 n \pi)^{\frac{n}{2}}}{ \big( \det( \sum_{i=1}^n A_i) \big)^{\frac{1}{2}}} \ \prod_{i=1}^{n} \left[  c_i \ f_i \left( \frac{\det(A_i) }{c_i^{2}} \right) \right]^{\frac{1}{n}}.
\end{equation}
In particular, if $A_i=A$ for all $i$, where $A$ is a $(n \times n)$  positive definite matrix, then (\ref{exponential1}) becomes
\begin{equation}\label{exponential2}
D_{\vec f} \left(P_{\vec \varphi} , Q_{\vec \varphi} \right)\ = \  \frac{(2 \pi)^{n/2}}{\sqrt{\det (A)}} \ \prod_{i=1}^{n} \left[ c_i\  f_i \bigg(  \frac{ \det (A)}{c_i^2} \bigg) \right]^{\frac{1}{n}}.
\end{equation}
\vskip 2mm
\begin{prop}\label{cor:mixed-f-div-SLn}
For $1 \leq i \leq n$, let $f_i: (0, \infty) \rightarrow \mathbb{R}_+$ be  convex and/or concave functions and 
let $ \varphi_i:\R^{n}\rightarrow [0, \infty)$ be  
log-concave functions.  
Then $D_{\vec f} \left(P_{\vec \varphi} , Q_{\vec \varphi} \right) $
is   invariant  under self adjoint $SL(n)$ maps.
\end{prop}
\begin{proof}
Let $A: \mathbb{R}^n \rightarrow \mathbb{R}^n $ be a self adjoint,  $SL(n)$ invariant  linear map.
\begin{eqnarray*}
&& D_{\vec f} \left((P_{\varphi_1\circ A}, \cdots,P_{\varphi_n\circ A}) , (Q_{\varphi_1\circ A}, \cdots, Q_{\varphi_n\circ A} ) \right) \\
&& = {\displaystyle \int }  \prod_{i=1}^{n} \left[ \varphi_i( Ax)  \  f_i \left(
\frac{e^{\frac{\langle\grad\varphi_i(Ax), x \rangle}{\varphi_i(Ax)}} }{(\varphi_i(Ax))^2} \  \mbox{det} \left[ \text{Hess} \left( -\log \varphi_i(Ax) \right)\right] \right) \right]^{\frac{1}{n}} dx \\
&& = \frac{1}{ | \det A| }  {\displaystyle \int }     \prod_{i=1}^{n} \left[ \varphi_i  \  f_i \left( (\det A)^2
\frac{e^{\frac{\langle\grad\varphi_i, x \rangle}{\varphi_i}} }{\varphi_i^2} \  \mbox{det} \left[  \text{Hess} \left( -\log \varphi_i \right)\right] \right) \right]^{\frac{1}{n}} dx \\
&& = D_{\vec f} \left((P_{\varphi_1}, \cdots,P_{\varphi_n}) , (Q_{\varphi_1}, \cdots, Q_{\varphi_n} ) \right) .
\end{eqnarray*}
\end{proof}
\
Recall that  for  a function $ \varphi :\R^{n}\rightarrow [0, \infty) $,  the dual function $\varphi^\circ$ \cite{ArtKlarMil}  is defined by  
$$ \ \varphi^\circ (y) = \inf_{x \in \R^n } \left[ \frac{e^{- \langle x,y \rangle}}{\varphi(x)}  \right].$$
If $ \varphi $ is a  log concave function, i.e.,  $\varphi (x)= e^{-\psi (x)}$ with $\psi : \R^{n} \rightarrow \R  $ convex, then 
this duality notion is connected with the Legendre transform 
$  \psi^* (y) = \sup_{x \in \R^n} \left[ \langle x,y \rangle  -\psi(x) \right] $, 
\begin{equation}\label{Polare}
 \varphi^\circ (y) =  e^{- \psi^* (y) }. 
\end{equation}
\vskip 2mm
For  special forms of the log concave functions $\varphi_i$ we have the following duality formula. This is the functional counterpart to the one proved in \cite{Ye2012}  for convex bodies and for special $f$.
\vskip 2mm

\begin{theo}\label{f-duality}
For $1 \leq i \leq n$, let  $f_i:(0,  \infty) \rightarrow \mathbb{R}_+$ be  convex and/or concave  functions and let
$\varphi_i =  \lambda_i \varphi$,  for some log concave function $ \varphi :\R^{n}\rightarrow [0, \infty)$ and $\lambda_i >0$. Then
\begin{equation}\label{mixed-polar-identity} 
 D_{\vec f} \left(P_{\vec{ \varphi^\circ}},  Q_{\vec{ \varphi^\circ}}\right) = D_{\vec f^*} \big(P_{\vec\varphi},  Q_{\vec\varphi} \big) .
\end{equation}
\end{theo}
\vskip 2mm
\begin{proof}
We write  $ \varphi = e^{- \psi}$, $\psi$ convex,  and let  $\psi^* (y)$ be the Legendre transform of $\psi$.
Please note that when $\psi$ is a $C^2$ strictly convex function, then 
$$
\psi(x) + \psi^* (y) = \langle x, y \rangle
\hbox{ if and only if }
y = \nabla \psi(x)
\hbox{ if and only if }
x = \nabla \psi^* (y).
$$
It follows that 
\begin{equation}
\label{dualitybis}
\forall y \in \R^n, \ \psi(\nabla \psi^* (y)) = \langle y , \nabla \psi^* (y) \rangle -  \psi^* (y)
\end{equation}
and
\begin{equation}
\label{dualitygrad}
\nabla \psi \circ \nabla \psi^*  = \nabla \psi^* \circ \nabla \psi = {\text Id}, 
\end{equation}
so that for any $x,y \in \R^n$,
\begin{equation}
\label{dualityhess}
\text{Hess} \, \psi (\nabla \psi^* (y) ) \ \text{Hess} \, \psi^* (y) = {\text Id} = \text{Hess} \, \psi^*(\nabla \psi(x)) \ \text{Hess} \,  \psi(x).
\end{equation}
Using equations (\ref{dualitybis}),  (\ref{dualitygrad}) and (\ref{dualityhess}), the change of variable $x=\nabla \psi^* (y) $ gives
\begin{eqnarray*}
&& \hskip -0.7cm  D_{\vec f^*} \left(P_{\vec\varphi},  Q_{\vec\varphi} \right) \\ 
 &=&  {\displaystyle \int }  \prod_{i=1}^{n} \left[ \varphi_i  \  f_i^* \left(
\frac{e^{\frac{\langle\grad\varphi_i, x \rangle}{\varphi_i}} }{\varphi_i^2} \  \mbox{det} \left[  \text{Hess} \left( -\log \varphi_i \right)\right] \right) \right]^{\frac{1}{n}} dx \\
&=& {\displaystyle \int }  \prod_{i=1}^{n} \left[ \frac{  e^{\frac{\langle\grad\varphi, x \rangle}{\varphi}} \   \mbox{det} \left[  \text{Hess} \left( -\log \varphi \right)\right]}{\lambda_i \varphi } f_i \left( \frac{ \lambda_i^2 \varphi^2   e^{- \frac{\langle\grad\varphi, x \rangle}{\varphi}} }{  \mbox{det} \left[  \text{Hess} \left( -\log \varphi \right) \right]} \right)\right]^{\frac{1}{n}} dx \\
&=& \frac{1}{ \left(\lambda_1 \cdots \lambda_n \right)^{\frac{1}{n}} } \ {\displaystyle \int }  \prod_{i=1}^{n} \left[  \det \left( \text{Hess} \psi (x)\right)  \  e^{ \psi (x) - \langle \nabla \psi , x\rangle}  f_i \left( \frac{ \lambda_i^2 \ e^{ -2  \psi (x) +  \langle \nabla \psi, x\rangle} }{ \det\left(\text{Hess} \psi (x)\right)  }  \right)  \right]^{\frac{1}{n}}   \hskip -3mm dx \\
&=& \frac{1}{ \left(\lambda_1 \cdots \lambda_n \right)^{\frac{1}{n}} } \ {\displaystyle \int}  \prod_{i=1}^{n} \left[ \det\left( \text{Hess} \psi ( \nabla \psi^* (y))\right)   \  e^{ (\psi ( \nabla \psi^* (y) ) - \langle y, \nabla ( \psi^* (y) )  \rangle ) }  \right]^{\frac{1}{n}} \\
&& \hskip 2.5cm  \times \prod_{i=1}^{n} \left[  f_i \left( \frac{ \lambda_i^2 \ e^{ -2  \psi( \nabla \psi^* (y) ) +  \langle y, \nabla \psi^* (y) \rangle} }{  \det\left(\text{Hess} \psi ( \nabla \psi^* (y) )\right)}\right) \right]^{\frac{1}{n}} \   \det\left(\text{Hess} \psi^* (y)  \right)  dy  \\
&=& \frac{1}{ \left(\lambda_1 \cdots \lambda_n \right)^{\frac{1}{n}} } \ {\displaystyle \int } \prod_{i=1}^{n} \left[  e^{ -  \psi^* (y) }   \  f_i \left( \lambda_i^{2} \ \det(\text{Hess} \psi^* (y)  ) \ e^{ -  \langle y , \nabla \psi^* (y) \rangle +  2 \psi^* (y) }  \right)  \right]^{\frac{1}{n}} dy  \\
&=& \frac{1}{ \left(\lambda_1 \cdots \lambda_n \right)^{\frac{1}{n}} } \ {\displaystyle \int}  \prod_{i=1}^{n} \left[  \varphi^\circ \ f_i \left( \lambda_i^{2} \ \mbox{det}\left[ \text{Hess} \left(-\log  \varphi^\circ  \right) \right] \ \frac{
 e^{ \frac{\langle \nabla \varphi^\circ, x\rangle}{ \varphi^\circ}}  }{ (\varphi^\circ)^{2  }}   \right) \right]^{\frac{1}{n}}  dx \\
&=& D_{\vec f} \left(P_{\vec{ \varphi^\circ}},  Q_{\vec{ \varphi^\circ}}\right).
\end{eqnarray*}
The last part follows from the fact that $ (\lambda \varphi)^\circ = \frac{\varphi^\circ}{\lambda}  $,
for $\lambda \in \R$, $\lambda \ne 0$.
\end{proof}
\par
\noindent
\textbf{Remark.}
 If $f_i = f$ and  $\lambda_i = 1$,  i.e. $\varphi_i = \varphi $ for all $i =1, \cdots, n$, then  $D_f (P_{\varphi^\circ} , Q_{\varphi^\circ}) = D_{f^*} (P_{\varphi} , Q_{\varphi}) $. This was proved in \cite{CaglarWerner}.

\vskip 3mm
The classical Alexandrov-Fenchel inequality for mixed volumes of convex bodies is one of the most important results in convex geometry.
We refer to e.g., \cite{Sch} for the details and prove  now an Alexandrov-Fenchel type inequality
for mixed $f$-divergences for log concave functions. The proof is similar to  one given in \cite{WernerYe2013}. We include it for completeness.
We use the following notations. 
\par
For  $1 \leq m \leq  n-1$  and $k > n-m$, we put 
$$
 \vec f_{m,k}  = (f_1, f_2, \cdots , f_{n-m}, \underbrace{f_{k}, \cdots , f_{k}}_m), 
 $$ 
$$
P_{\vec \varphi_{m,k}}=   (P_{\varphi_1},  \cdots ,P_{\varphi_{n-m}}, \underbrace{P_{\varphi_{k}}, \cdots , P_{\varphi_{k}}}_m), \hskip 3mm 
 Q_{\vec \varphi_{m,k}}=   (Q_{\varphi_1},  \cdots ,Q_{\varphi_{n-m}}, \underbrace{Q_{\varphi_{k}}, \cdots , Q_{\varphi_{k}}}_m).
$$
\par
Following  \cite{HardyLittlewoodPolya}, we say that two functions $f$ and $g$ are {\em effectively proportional}  if there are 
constants $a$ and $b$, not both zero, such that $af=bg$. Functions $f_1, \dots, f_m$ are effectively proportional if every pair $(f_i, f_j), 1\leq i, j\leq m$ is  effectively proportional.  A null function is effectively proportional to any function.
\par
Moreover, for  $1 \leq  m \leq n-1$,  we let
\begin{equation}\label{h0}
h_0(x) = \prod_{i=1}^{n-m} \left[ \varphi_i  \  f_i \left(
\frac{e^{\frac{\langle\grad\varphi_i, x \rangle}{\varphi_i}} }{\varphi_i^2} \  \mbox{det} \left[  \text{Hess} \left( -\log \varphi_i \right)\right] \right) \right]^{\frac{1}{n}}
\end{equation}
and for $j = 0, \cdots, m-1$,
\begin{equation} \label{hj}
h_{j+1}(x) =  \left[ \varphi_{n-j}  \  f_{n-j} \left(
\frac{e^{\frac{\langle\grad\varphi_{n-j}, x \rangle}{\varphi_{n-j}}} }{\varphi_{n-j}^2} \  \mbox{det} \left[  \text{Hess} \left( -\log \varphi_{n-j} \right)\right] \right) \right]^{\frac{1}{n}}.
\end{equation}
\vskip 3mm
Then an Alexandrov-Fenchel type inequality holds for log concave functions, namely,
\vskip 2mm
\begin{theo}\label{Alexandrov-Fenchel}
For $1 \leq i \leq n$,  let $f_i: (0, \infty) \rightarrow \mathbb{R}_+$ be  either all convex or all concave functions and 
let $ \varphi_i:\R^{n}\rightarrow [0, \infty)$ be  
log-concave functions.  Then, for $ 1 \leq m \leq n-1$,
\begin{eqnarray*}
 \left[ D_{\vec f} \left( P_{\vec \varphi}, Q_{\vec \varphi} \right) \right]^m  \leq  
 \prod_{k=n-m+1}^{ n }  D_{ \vec f_{m,k }}  \left( P_{\vec \varphi_{m,k}},  Q_{\vec \varphi_{m,k}}  \right).
\end{eqnarray*}
Equality holds if and only if one of the functions $h_0 ^\frac{1}{m} h_j$, $1 \leq j \leq m$,   is null or all are effectively proportional.
\par
If $m=n$, then
\begin{eqnarray*}
 \left[ D_{\vec f} \left( P_{\vec \varphi}, Q_{\vec \varphi} \right) \right]^n  \leq  
 \prod_{k=1}^{ n }  D_{  f_{k }}  \left( P_{ \varphi_{k}},  Q_{ \varphi_{k}}  \right).
\end{eqnarray*}
Equality holds if and only if one of the functions $h_j$, $1 \leq j \leq n$,   is null or all are effectively proportional.
\end{theo}
\vskip 2mm 
\noindent
\textbf{Remark.}
In particular, equality holds in Theorem \ref{Alexandrov-Fenchel},  if
(i)  all $\varphi_i$ coincide and $f_i = \lambda_i f$ for some positive convex function $f$ and $\lambda_i >0$, $i = n-m +1, \cdots ,n $, or 
(ii) $f_i = \lambda_i f$, for some positive convex function $f$, for some $ \lambda_i >0$, $\varphi_i = a_i \varphi$, for some positive,  log concave function $\varphi$,
 for some $a_i$, $i = n-m +1, \cdots ,n $ and $f$ is homogeneous of degree $ \alpha \in [0,1)$.
\vskip 2mm
\begin{proof}
We first treat the case $m=n$. By  H\"older's inequality,  e.g.,  \cite{HardyLittlewoodPolya}, 
\begin{eqnarray*}
  D_{\vec f} \left( P_{\vec \varphi}, Q_{\vec \varphi} \right)  &=&  {\displaystyle \int }   \prod_{i=1}^{n} \left[ \varphi_i  \  f_i \left(
\frac{e^{\frac{\langle\grad\varphi_i, x \rangle}{\varphi_i}} }{\varphi_i^2} \  \mbox{det} \left[  \text{Hess} \left( -\log \varphi_i \right)\right] \right) \right]^{\frac{1}{n}} dx\\
&\leq& \prod_{i=1}^{n}  \   \left[ {\displaystyle \int }   \varphi_i  \  f_i \left(
\frac{e^{\frac{\langle\grad\varphi_i, x \rangle}{\varphi_i}} }{\varphi_i^2} \  \mbox{det} \left( \text{Hess} \left( -\log \varphi_i \right) \right)\right) \right]^{\frac{1}{n}} dx\\
&= &
 \prod_{i=1}^{ n} \left( D_{  f_i }  \left( P_{\varphi_i} , Q_{\varphi_i}  \right)\right)^\frac{1}{n} .
\end{eqnarray*}
\par
Let now $ m \leq n-1$. 
Again, by H\"older's inequality, 
\begin{eqnarray*}
  \left[ D_{\vec f} \left( P_{\vec \varphi}, Q_{\vec \varphi} \right) \right]  ^m  &=& \left( \int \prod_{j=0}^{m-1} \left( h_0(x) h_{j+1}^m(x)\right)^\frac{1}{m} dx \right)^m \\
 &\leq & \prod_{j=0}^{m-1} \left( \int  h_0(x)  h_{j+1}^m (x) \right) dx 
 \ = \  \prod_{k=n-m+1}^{ n}  D_{ \vec f_{m,k }}  \left( P_{\vec \varphi_{m,k}},  Q_{\vec \varphi_{m,k}}  \right).
\end{eqnarray*}
In both cases, characterization of equality follows from the equality characterization in H\"older's inequality e.g.,  \cite{HardyLittlewoodPolya}.
\end{proof}

\vskip 3mm
The  following  entropy inequality is a consequence of Theorem \ref{Alexandrov-Fenchel}.
\vskip 2mm
\noindent
\begin{theo}\label{mixed-entropy}
For $1 \leq i \leq n$, let $f_i: (0, \infty) \rightarrow \mathbb{R}_+$ be concave functions and 
 let $ \varphi_i:\R^{n}\rightarrow [0, \infty)$ be log-concave functions.
Then
\begin{eqnarray}\label{ineq:mixed}
 \left[ D_{\vec f} \left(P_{\vec \varphi} , Q_{\vec \varphi} \right)  \right]^n  \ \leq \  
 \prod_{i=1}^{ n}  \ f_i \left(  \frac{\int \varphi_i^\circ dx}{ \int \varphi_i dx}   \right)
 \  \left( \int_{} \varphi_i  dx \right) .
\end{eqnarray}
Equality holds  if  and only if
 $\varphi_i(x)=c_i e^{-\frac{1}{2} \langle A x, x \rangle}$  where  $c_i$  is a positive constant and $A$ is a $(n \times n)$  positive definite  matrix.
 \end{theo}
\vskip 2mm
\begin{proof}
The inequality follows immediately from Theorem \ref{Alexandrov-Fenchel} for $m=n$  and Theorem 1 in \cite{CaglarWerner}, which says  that for a concave function  $f:(0, \infty) \rightarrow \mathbb{R}$ and 
a log-concave function $ \varphi:\R^{n}\rightarrow [0, \infty)$, we have 
\begin{eqnarray}\label{thm0,1}  
D_f(P_\varphi, Q_\varphi)  \leq \ f \left(  \frac{\int \varphi^\circ dx}{ \int \varphi dx}   \right)
 \  \left( \int_{} \varphi  dx \right).
 \end{eqnarray}
It was proved  in  \cite{Werner-Yolcu}, that equality  holds in (\ref{thm0,1}) if and only if $\varphi(x)=c e^{-\frac{1}{2} \langle A x, x \rangle}$ where $c>0$ is a constant and $A$ is a $(n \times n)$  positive definite matrix.
\par
We now treat the equality characterization. 
Using (\ref{exponential2}),  it is easy to check that equality holds if  $\varphi_i(x)=c_i e^{-\frac{1}{2} \langle A x, x \rangle}$, $1 \leq i \leq n$.
On the other hand, if equality holds in (\ref{ineq:mixed}), then in particular, 
\begin{equation*}
 \prod_{i=1}^{ n} D_{ f_i} \left(P_{ \varphi_i} , Q_{\varphi_i} \right) =
 \prod_{i=1}^{ n}  \ f_i \left(  \frac{\int \varphi_i^\circ dx}{ \int \varphi_i dx}   \right)
 \  \left( \int_{} \varphi_i  dx \right) .
 \end{equation*}
Thus,  equality holds in particular for all $i$ in the entropy inequality (\ref{thm0,1}),  which,  by the equality characterization of \cite{Werner-Yolcu},  means that 
for all $i$, $\varphi_i(x) =  c_i e^{-\frac{1}{2} \langle A_i x, x \rangle}$, where
 $c_i$ is a positive constant  and $A_i$ is a $(n \times n)$  positive definite  matrix. Thus, also using  (\ref{exponential1}),  the equality condition leads to the following identity
 \begin{equation}\label{gleich}
  \prod_{i=1}^{ n}  \left(  \det (A_i ) \right) ^\frac{1}{n} = \frac{\det \left(\sum_{i=1}^n A_i\right) }{n^n}.
\end{equation} 
The Brunn Minkowski inequality for matrices \cite{CoverThomas2006, KyFan, MINKOWSKI} says that for positive definite matrices $A_i$, $1 \leq i \leq n$,  one has
 \begin{equation}\label{BMI}
 \left(  \det \left( \sum_{i=1}^n A_i \right) \right) ^\frac{1}{n} \geq \sum_{i=1}^n \left(\det(A_i\right)^\frac{1}{n},
\end{equation} 
with equality  if and only if  all  $A_i = \lambda _i A$ for some positive definite matrix $A$ and scalars $\lambda_i \geq 0$.
It follows from the geometric arithmetic mean inequality that 
\begin{equation*}
\left(\sum_{i=1}^n \left(\det(A_i)\right)^\frac{1}{n}\right)^n \geq n^n  \prod_{i=1}^{ n}  \left(  \det (A_i ) \right) ^\frac{1}{n},
\end{equation*}
with equality if and only if $\det(A_i) = \det(A_j)$, for all $i,j$.
With (\ref{BMI}),  we get altogether, 
\begin{equation} \label{BM+GA}
\prod_{i=1}^{ n}  \left(  \det (A_i ) \right) ^\frac{1}{n} \leq \frac{1}{n^n} \   \left(\sum_{i=1}^n \left(\det(A_i)\right)^\frac{1}{n}\right)^n \leq   \frac{1}{n^n} \   \left(  \det \left( \sum_{i=1}^n A_i \right) \right).
\end{equation}
By assumption, equality  (\ref{gleich}) holds. Therefore, we have equality in both, the geometric arithmetic mean inequality and the Brunn Minkowski inequality which means that for all 
$i$, $A_i= \lambda A$, for some $\lambda>0$, for some positive definite matrix $A$. Hence we have that $\varphi_i(x) = c_i e^{ - \langle A x , x \rangle / 2  }$.
\end{proof}
If we let $f_i(t) = \log t$, $1 \leq i \leq n$,  in Theorem \ref{mixed-entropy}, then  we obtain the following corollaries. We use again the notation  $a_+ = ( \max \{a_1,0\}, \max \{a_2,0\}, \cdots, \max \{a_n,0\}  )$,  for $a \in \R^n$.
\vskip 2mm
\begin{cor}\label{Korrolar1}
For $1 \leq i \leq n$, let $ \varphi_i:\R^{n}\rightarrow [0, \infty)$ be 
log-concave
functions.
Then
\begin{eqnarray}\label{eqn:Korollar1}
 \left[ D_{KL} \left(P_{\vec \varphi} , Q_{\vec \varphi} \right) \right]^{n} \leq 
  \prod_{i=1}^{ n} \ \left[\log \left(  \frac{\int \varphi_i^\circ dx}{ \int \varphi_i dx}   \right)\right]_+
 \  \left( \int \varphi_i  dx \right) .
\end{eqnarray}
Equality holds if and only if $\varphi_i (x)= c_i e^{-\frac{1}{2} \langle A x, x \rangle}$, where
 $c_i$ is a positive constant  and $A$ is a $(n \times n)$  positive definite  matrix.
\end{cor}
\vskip 3mm
\begin{cor}\label{Korrolar2}
For $1 \leq i \leq n$, let $ \varphi_i:\R^{n}\rightarrow [0, \infty)$ be 
log-concave
functions such that $\int x \varphi_i dx =0$ for all $i$.
Then
\begin{eqnarray}\label{eqn:Korollar2}
 \left[ D_{KL} \left(P_{\vec \varphi} , Q_{\vec \varphi} \right)  \right]^{n} \leq 
  \prod_{i=1}^{ n} \ \left[\log \left(  \frac{(2 \pi)^n}{\left( \int \varphi_i dx\right)^2}   \right)\right]_+
 \  \left( \int_{} \varphi_i  dx \right) .
\end{eqnarray}
Equality holds if and only if $\varphi_i (x)= c_i e^{-\frac{1}{2} \langle A x, x \rangle}$, where
 $c_i$ is a positive constant  and $A$ is a $(n \times n)$  positive definite  matrix.
\end{cor}
\vskip 2mm
\begin{proof} The functional Blaschke Santal\'o inequality  \cite{ArtKlarMil, KBallthesis,  MeyerFradelizi2007, Lehec2009}  says that for a log concave function $\varphi$
with barycenter at $0$, i.e.,   $\int x \varphi dx =0$, one has 
$$
\left( \int \varphi\, dx \right) \   \left( \int  \varphi^\circ \, dx \right) \leq ( 2 \pi )^n ,
$$
with equality if and only if there exists a positive definite matrix $A$ and $c>0$ such that 
$\varphi (x) = c  e^{ - \langle A x , x \rangle / 2  }$. We apply the functional Blaschke Santal\'o  inequality on the right hand side of (\ref{eqn:Korollar1}) to each $\varphi_i$
and get  inequality  (\ref{eqn:Korollar2}).
\par
Using (\ref{exponential2}), it is easy to see that equality holds in (\ref{eqn:Korollar2})  if $\varphi_i (x)= c_i e^{-\frac{1}{2} \langle A x, x \rangle}$, where
 $c_i$ is a positive constant  and $A$ is a $(n \times n)$  positive definite  matrix.
On the other hand, if equality holds in (\ref{eqn:Korollar2}), then equality holds in particular for all $i$ in the functional Blaschke Santal\'o inequality which means that 
for all $i$, $\varphi_i(x) =  c_i e^{-\frac{1}{2} \langle A_i x, x \rangle}$, where
 $c_i$ is a positive constant  and $A_i$ is a $(n \times n)$  positive definite  matrix. Thus, as above in the proof of Theorem \ref{mixed-entropy},  the equality condition again leads to the  identity
 \begin{equation*}
  \prod_{i=1}^{ n}  \left(  \det (A_i ) \right) ^\frac{1}{n} = \frac{\det \left(\sum_{i=1}^n A_i\right) }{n^n}
\end{equation*} 
and we conclude as above.
\end{proof}

\vskip 3mm
\section{ The $i$-th mixed $f$-divergence for log-concave functions}\label{Section-i-mixedlogcon}
Throughout this section, let $f_1, f_2 : (0, \infty) \rightarrow  \R_+ $ be either convex or concave functions. 
As above, let $(X, \mu)$ be a finite measure space  and, for $l=1,2$, let  $P_l=p_l \mu$ and  $Q_l=q_l \mu$ be  measures on $X$ that are  absolutely continuous with respect to the measure $\mu$. 
Denote $\vec f  = (f_1, f_2$),  $ \vec P = (P_1, P_2)$ and $ \vec Q =( Q_1, Q_2)$. 
\vskip 2mm
The  $i$-th mixed $f$-divergence was introduced  in  \cite{WernerYe2013}. We  refer to  \cite{WernerYe2013} for properties and examples and only give the definition.
\begin{defn} \label{DW}
Let $i \in \R$. The $i$-th mixed $f$-divergence for $(\vec f, \vec{\textbf{P}}, \vec{\textbf{Q}})$ is defined in \cite{WernerYe2013} as 
\begin{equation}\label{def:mixed-fdiv}
D_{\vec f}(\vec{\textbf{P}}, \vec{\textbf{Q}};i )=
 \int_{X}  \left[f_1 \left(\frac{p_1}{q_1} \right) q_1 \right]^{\frac{i}{n}} \left[f_2 \left(\frac{p_2}{q_2} \right) q_2 \right]^{\frac{n-i}{n}}d\mu.
\end{equation}
\end{defn}
As before, for $l =1,2 $, we  let
\begin{equation}\label{Q,P1}
q_{\varphi_l}=\varphi_l  \hskip 4mm \text{and} \hskip 4mm   p_{\varphi_l} = \varphi_l^{-1}  e^{\frac{\langle\grad\varphi_l, x \rangle}{\varphi_l}} \mbox{det} \left[  \text{Hess} \left(-\log \varphi_l \right)\right]
\end{equation}
and  use Definition \ref{DW} with   $q_l=q_{\varphi_l}$ and $p_l = p_{\varphi_l}$, $l=1,2$,  and get 
the {\em $i$-th mixed  $f$-divergences} for log concave  functions.
\vskip 2mm
\begin{defn} \label{i-th-mixed-f-div-log}
Let $f_1, f_2 : (0, \infty) \rightarrow  \R_+ $ be either convex or concave functions and let $ \varphi_1, \varphi_2:\R^{n}\rightarrow [0, \infty)$ be  log concave functions.
Let $i \in \R$. Then the  $i$-th mixed $f$-divergence of $\vec \varphi = (\varphi_1, \varphi_2)$ is
\begin{eqnarray*}\label{def:i-th-mixed-fdiv}
&& \hskip -5mm D_{\vec f}\left((P_{\varphi_1}, P_{\varphi_2}) , (Q_{\varphi_1}, Q_{\varphi_2} ); i \right) =
\int   \left[f_1 \left(\frac{p_1}{q_1} \right) q_1 \right]^{\frac{i}{n}} \left[f_2 \left(\frac{p_2}{q_2} \right) q_2 \right]^{\frac{n-i}{n}} dx \   =  \nonumber  \\
&& \hskip -5mm {\displaystyle \int }   \hskip-1mm \left[ \varphi_1    f_1 \left(
\frac{e^{\frac{\langle\grad\varphi_1, x \rangle}{\varphi_1}} }{\varphi_1^2}  \mbox{det} \left[\text{Hess} \left( -\log \varphi_1 \right)\right] \right) \right]^{\frac{i}{n}} 
\hskip-1mm \left[ \varphi_2    f_2 \left(
\frac{e^{\frac{\langle\grad\varphi_2, x \rangle}{\varphi_2}} }{\varphi_2^2}   \mbox{det} \left[\text{Hess} \left( -\log \varphi_2 \right)\right] \right) \right]^{\frac{n-i}{n}} 
\hskip -7mm dx.
\end{eqnarray*}
\end{defn}
\vskip 3mm
If we let  $q_l=q_{\varphi_l}$ and  $p_l= p_{\varphi_l}$,  $l=1,2$, then the  following proposition is an immediate consequence of Proposition V.I of \cite{WernerYe2013}.  
We also denote
$$  P_{ \vec \varphi }= (P_{\varphi_1},   P_{\varphi_2}),  \hskip 4mm   Q_{ \vec \varphi }= ( Q_{\varphi_1},  Q_{\varphi_2}).
$$
\vskip 2mm
\begin{prop}\label{i-th-mixed-prop}
Let $f_1, f_2 : (0, \infty) \rightarrow  \R_+ $ be either convex or concave functions and let $ \varphi_1, \varphi_2:\R^{n}\rightarrow [0, \infty)$ be  log concave functions. 
If $j \leq i \leq k$ or $k \leq i \leq j$, then
\begin{eqnarray*}\label{def:i-th-mixed-fdiv}
D_{\vec f}  \left(  P_{ \vec \varphi } \ ,  Q_{ \vec \varphi }\  ;  \ i \right)  \  \leq  \  \left[ D_{\vec f}  \left(  P_{ \vec \varphi } \ ,  Q_{ \vec \varphi } \  ; \ j \right) \right]^{\frac{k-i}{k-j}} \times  \  \left[ D_{\vec f}  \left(  P_{ \vec \varphi } \ ,  Q_{ \vec \varphi }  \ ; \ k \right) \right]^{\frac{i-j}{k-j}} .
\end{eqnarray*}
Equality holds trivially if $ i=k$ or $i=j$. Otherwise, equality holds if and only if one of the functions $f_l\left(\frac{p_{\varphi_l}}{q_{\varphi_l}} \right) \ q_{\varphi_l}$, $l=1,2$ is null or 
are effectively proportional.
\end{prop}
 \vskip 3mm
The next  corollary follows immediately from Proposition \ref{i-th-mixed-prop} and (\ref{thm0,1}).
\vskip 2mm
\begin{cor}\label{corollary-i1}
Let 
$ \varphi_1, \varphi_2:\R^{n}\rightarrow [0, \infty)$ be  log concave functions and let  $f_1, f_2 : (0, \infty) \rightarrow  \R_+ $. If
 $f_1,  f_2 $ are  concave  and  $ 0 \leq i \leq n$, then
\begin{eqnarray*}\label{def:i-th-mixed-fdiv}
&& \hskip -4mm \left[ D_{\vec f}  \left(  P_{ \vec \varphi } \ ,  Q_{ \vec \varphi }  \ ; \ i \right) \right]^n  \\
&&  \hskip 3mm  \leq \left[ f_1 \left(  \frac{\int \varphi_1^\circ dx}{ \int \varphi_1 dx}   \right)
 \  \left( \int_{} \varphi_1  dx \right) \right]^i   \times  \left[f_2 \left(  \frac{\int \varphi_2^\circ dx}{ \int \varphi_2 dx}   \right)
 \  \left( \int_{} \varphi_2  dx \right) \right]^{n-i} \hskip -3mm .
\end{eqnarray*}
\par
\noindent
If  (i) $f_1$ is convex,  $  f_2 $ is   concave and $i \geq n$, or 
(ii) $f_1$ is  concave,  $f_2 $ is convex  and $i \leq  0$, then the inequality is reversed.
\par
\noindent
Equality holds trivially if $ i=0$ or $i=n$. Otherwise, equality holds if  and only if $\varphi_l = c_l e^{- \frac{1}{2} \langle A x, x \rangle}$, $l=1,2$, where
 $c_l$ is a positive constant  and $A$ is a $(n \times n)$  positive definite  matrix. 
\end{cor}

\begin{proof} We give the proof in the first case. The others are done similarly.
Let $k=0$ and $j=n$ (or $j=0$ and $k=n$)  in Proposition \ref{i-th-mixed-prop}. By (\ref{thm0,1}),
\begin{eqnarray*}\label{def:i-th-mixed-fdiv}
&&\left[ D_{\vec f}  \left(  P_{ \vec \varphi } \ ,  Q_{ \vec \varphi }  \ ; \ i \right) \right]^n  \ \leq \ \left[ D_{ f_1}  \left(  P_{  \varphi_1 } \ ,  Q_{  \varphi_1 }  \right) \right]^{i} \times  \  \left[ D_{f_2}  \left(  P_{ \varphi_2 } \ ,  Q_{ \varphi_2 } \right) \right]^{n-i}  \\
&& \leq   \left[ f_1 \left(  \frac{\int \varphi_1^\circ dx}{ \int \varphi_1 dx}   \right)
 \  \left( \int_{} \varphi_1  dx \right) \right]^i   \  \left[f_2 \left(  \frac{\int \varphi_2^\circ dx}{ \int \varphi_2 dx}   \right)
 \  \left( \int_{} \varphi_2  dx \right) \right]^{n-i}.
\end{eqnarray*}
 \par

 It is easy to see that equality holds if $\varphi_l = c_l e^{- \frac{1}{2} \langle A x, x \rangle}$, $l=1,2$, where
 $c_l$ is a positive constant  and $A$ is a $ (n \times n)$  positive definite  matrix. 
 On the other hand, if equality holds  in the inequality, then in particular, equality holds in (\ref{thm0,1}), which   means that 
 $\varphi_l = c_l e^{- \frac{1}{2} \langle A_l x, x \rangle}$, $l=1,2$, where
 $c_l$ are  positive constants and $A_l$ are  $(n \times n)$  positive definite  matrices.  Thus, equality in the inequality leads to the following identity
 $$
 \det\left(\frac{i}{n} A_1 + \left(1-\frac{i}{n}\right) A_2 \right) = \left(\det A_1\right)^\frac{i}{n} \left(\det A_2\right)^{1-\frac{i}{n} }.
$$ 
We conclude again,  by the Brunn Minkowski inequality for matrices \cite{CoverThomas2006, KyFan, MINKOWSKI}, that $A_1=A_2$.
\end{proof}
\noindent
\textbf{Remark.}
In particular, if we let $f_1 (t) = f_2 (t) = \log (t)$ in Corollary \ref{corollary-i1}, 
then we obtain similar results for the {\em $i$-th mixed Kullback-Leibler divergence}, as in Corollary \ref{Korrolar1}.

\vskip 3mm

\section{Applications to special functions:  Mixed $L_\lambda$-affine surface area}\label{Section-specialf}

Now we consider special functions $f$ and obtain special cases of mixed $f$-divergences for log concave functions.  
\par
For $i=1, \dots, n$, we let  $f_i(t)=t^{\lambda}$, $\ - \infty < \lambda< \infty$, and  we obtain the {\em mixed $L_\lambda$-affine surface area},  denoted by   $as_{ \lambda} (\vec \varphi)$, for    log concave functions $ \varphi_i $, 
\begin{equation}\label{mixed-asp-Logconcave}
as_{ \lambda} (\vec \varphi) = 
{\displaystyle \int }   \prod_{i=1}^n  \left[ \varphi_i \   \left(   
\frac{e^{\frac{\langle\grad\varphi_i, x \rangle}{\varphi_i}}}{\varphi_i^2}  \  \mbox{det} \left[   \text{Hess} \left(- \log \varphi_i \right)\right] \right)^{\lambda} \right]^{\frac{1}{n}}dx,
\end{equation}
or, writing $\varphi_i(x)= e^{-\psi_i(x)}$, $\psi_i$ convex,
\begin{equation}\label{mixed-sp-Logconcave1}
as_{ \lambda} (\vec \varphi)  = \int  \prod_{i=1}^n  \left[   e^{(2\lam -1)\psi_i(x)-\lam \langle x, \nabla\psi_i(x)\rangle}\left(\det \, {\text{Hess } \psi_i (x)}\right)^{\lam} \right]^{\frac{1}{n}}dx.
\end{equation}
\vskip 2mm
In particular,  $as_{0}(\vec \varphi)= \int  (\varphi_1  \cdots  \varphi_n)^{\frac{1}{n}}  dx $.
Please note   that for any $\vec \varphi$, we have 
$as_{ \lambda} (\vec \varphi) \geq 0 $.
Moreover,  by Proposition \ref{cor:mixed-f-div-SLn}, the  $as_{  \lambda} (\vec \varphi)$ are   invariant under self adjoint $SL(n)$ maps.
\vskip 2mm
\noindent
\textbf{Remarks.}
(i) If we let $\varphi_i = \varphi$ for $ i = 1, \cdots, n$,  we recover the {\em  $L_\lambda$-affine surface area}, $as_{ \lambda } ( \varphi)$, defined in \cite{CFGLSW} (see also \cite{CaglarWerner}),
\begin{equation}\label{asp-Logconcave}
as_{ \lambda} ( \varphi) = 
{\displaystyle \int }     \varphi \   \left(   
\frac{e^{\frac{\langle\grad\varphi, x \rangle}{\varphi}}}{\varphi^2}  \  \mbox{det} \left[   \text{Hess} \left(- \log \varphi \right)\right] \right)^{\lambda} dx.
\end{equation}
(ii)  For  $1 \leq i \leq n $, let $A_i$ be a $(n \times n)$  positive definite matrix, $c_i>0$ a constant and let $ \  \varphi_i(x) = c_i e^{- \frac{1}{2} \langle A_ix, x \rangle}$. Then, 
\begin{equation}\label{affineGaussian}
  as_{ \lambda} (\vec \varphi) = \  \frac{( 2 n \pi)^{\frac{n}{2}}}{ \big( \det( \sum_{i=1}^n A_i) \big)^{\frac{1}{2}}} \  \prod_{i=1}^{n} \left[  c_i^{1 - 2\lambda} \  \left( \det(A_i) \right)^{\lambda} \right]^{\frac{1}{n}} .
\end{equation}

We also give a definition for  $ as_{\infty } (\vec \varphi)$ and  $ as_{-\infty } (\vec \varphi)$, similarly as it was done for the $L_\lambda$-affine surface area \cite{CaglarWerner} (see also \cite{MW2}). 
\begin{equation}
as_{\infty } (\vec \varphi)= 
\max_{x} \ \prod_{i=1}^n  \left[ 
\frac{e^{\frac{\langle\grad\varphi_i, x \rangle}{\varphi_i}}}{\varphi_i^2}  \  \mbox{det} \left[   \text{Hess} \left(- \log \varphi_i \right)\right] \right]^{\frac{1}{n}} \quad \text{and} \quad 
as_{-\infty} (\vec  \varphi) =\frac{1}{ as_{\infty} (\vec \varphi)} .
\end{equation}
\vskip 3mm
\noindent
The following two propositions are  direct consequences of  Theorem \ref{f-duality} and Theorem \ref{Alexandrov-Fenchel}.
\vskip 2mm
\begin{prop}\label{lambda-duality}
Let $ \varphi_i:\R^{n}\rightarrow [0, \infty)$ be  
log-concave functions such that $\varphi_i = a_i \varphi$ for some log concave function $ \varphi :\R^{n}\rightarrow [0, \infty)$ and $ a_i >0$, $i=1, \cdots, n$. Then
\begin{equation}\label{mixed-affine-polar-identity} 
as_{ \lambda} \left( \vec \varphi \right) = as_{ 1- \lambda} \left( \vec{\varphi^\circ}\right) .
\end{equation}
\end{prop}
\vskip 2mm
Proposition \ref{lambda-duality} is generalization of the duality  $\ as_\lam(\varphi) = as_{1-\lam}(\varphi^\circ )$,  proved in \cite{CFGLSW}.

\vskip 3mm
In the next proposition we use, for $k >n-m$,  the notation
\begin{eqnarray*}\label{}
as_{ \lambda} ( \vec\varphi_{m,k})& = &
{\displaystyle \int }     \prod_{i=1}^{n-m} \left[ \varphi _i \   \left(   
\frac{e^{\frac{\langle\grad\varphi_i, x \rangle}{\varphi_i}}}{\varphi_i^2}  \  \mbox{det} \left[   \text{Hess} \left(- \log \varphi_i \right)\right] \right)^{\lambda}\right]^\frac{1}{n} \\
&&\hskip 25mm \left[ \varphi _k \   \left(   
\frac{e^{\frac{\langle\grad\varphi_k, x \rangle}{\varphi_k}}}{\varphi_k^2}  \  \mbox{det} \left[   \text{Hess} \left(- \log \varphi_k \right)\right] \right)^{\lambda}\right]^\frac{m}{n} dx.
\end{eqnarray*}
\vskip 2mm
\begin{prop}\label{prop: lambda mult}
For $1 \leq i \leq n$,  let $ \varphi_i:\R^{n}\rightarrow [0, \infty)$ be  
log-concave functions and let $- \infty < \lambda <  \infty$ . Then, if $ 1 \leq m \leq n-1$, 
\begin{eqnarray*}
 \left[ as_{\lambda} \left( \vec \varphi \right) \right]^m  \leq  
\prod_{k=n-m+1}^{ n }  as_{ \lambda} ( \vec \varphi_{m,k}).
\end{eqnarray*}
In particular, if $m=n$,
\begin{equation*}\label{multiplication-property}
\left[as_{\lambda} (\vec \varphi) \right]^{n} \leq \prod_{k=1}^n   as_{ \lambda } ( \varphi_k).
\end{equation*} 
\noindent
The equality characterization is the same as in Theorem \ref{Alexandrov-Fenchel}. 
\end{prop}

\vskip 3mm
Next, we prove affine isoperimetric inequalities for the  mixed $L_\lambda$-affine surface area. 
\begin{prop}
For  $1 \leq i \leq n$, let $ \varphi_i:\R^{n}\rightarrow [0, \infty)$ be  
log-concave functions such that $\varphi_i $ has barycenter at $0$.  
 If $ \lambda \in [0,1]$, then
\begin{equation}\label{isoperInequ}
\left[ \frac{as_{\lambda} (\vec \varphi) }{ as_{\lambda} (g, \cdots, g) } \right]^{n} \leq \ \prod_{i=1}^n \left( \frac{ \int \varphi_i}{ \int g} \right)^{1 -2 \lambda}   .
\end{equation} 
where $g(x) = e^{-\frac{\| x \|^2}{2}}$. Equality holds if  and only if $\varphi_i = c_i e^{ - \frac{1}{2} \langle Ax, x \rangle }$ where $c_i >0$, $1 \leq i \leq n$,  and $A$ is a $(n \times n)$  positive definite matrix.
\end{prop}

\begin{proof}
By  Proposition \ref{prop: lambda mult}, 
$$
\left[\frac{as_{\lambda} (\vec \varphi) }{as_{\lambda} (g, \cdots, g) } \right]^{n} \leq \ \prod_{i=1}^n  \frac{ as_{ \lambda } ( \varphi_i)}{ as_{\lambda} (g) } \ \leq \ \prod_{i=1}^n \left( \frac{ \int \varphi_i}{ \int g} \right)^{1 -2 \lambda} .
$$
The last part follows from a corollary  in \cite{CFGLSW},  which says  that for a log-concave function $ \varphi:\R^{n}\rightarrow [0, \infty)$ with barycenter at $0$,
\begin{eqnarray}\label{ineq-CFGLSW}
 \frac{ as_{ \lambda } ( \varphi)}{ as_{\lambda} (g) } \ \leq \  \left( \frac{ \int \varphi}{ \int g} \right)^{1 -2 \lambda}.
\end{eqnarray} 
It was proved  in  \cite{CFGLSW} that equality  holds  if and only if $\varphi(x)=c e^{-\frac{1}{2} \langle A x, x \rangle}$ where $c>0$ is a constant and $A$ is a $(n \times n)$  positive definite matrix.
\par
Using (\ref{affineGaussian}), it is easy to see that equality holds in (\ref{isoperInequ})  if $\varphi_i (x)= c_i e^{-\frac{1}{2} \langle A x, x \rangle}$, where
 $c_i$ is a positive constant  and $A$ is a $(n \times n)$  positive definite  matrix.
On the other hand, if equality holds in (\ref{isoperInequ}), then equality holds in particular, for all $i$, in the inequality (\ref{ineq-CFGLSW}) which means that 
for all $i$, $\varphi_i(x) =  c_i e^{-\frac{1}{2} \langle A_i x, x \rangle}$, where
 $c_i$ is a positive constant  and $A_i$ is a $(n \times n)$  positive definite  matrix. Thus, as before,  
the equality condition again leads to the  identity
 \begin{equation*}
  \prod_{i=1}^{ n}  \left(  \det (A_i ) \right) ^\frac{1}{n} = \frac{\det \left(\sum_{i=1}^n A_i\right) }{n^n}
\end{equation*} 
and we conclude as before.

\end{proof}

We also have a  Blaschke Santal\'o type inequality.
\vskip 2mm
\begin{prop}\label{duality-mixed-log}
For $1 \leq i \leq n$, let $ \varphi_i:\R^{n}\rightarrow [0, \infty)$ be  
log-concave functions such that $\varphi_i $ has barycenter at 0.  
 If $ \lambda \in [0,1]$, then
\begin{equation}\label{BS-mixed}
 as_{\lambda} (\vec{ \varphi})  as_{\lambda} ( \vec{\varphi^\circ}) \  \leq \ (2 \pi)^n    .
\end{equation} 
Equality holds if  and only if $\varphi_i = c_i e^{ -\frac{1}{2}  \langle Ax, x \rangle }$ where $c_i >0$, $1 \leq i \leq n$,   and $A$ is a $(n \times n)$  positive definite matrix.
\end{prop}

\begin{proof}
By Proposition \ref{prop: lambda mult}, 
$$\left[ as_{\lambda} ( \vec \varphi)  as_{\lambda} ( \vec{\varphi^\circ})  \right]^{n} \leq \prod_{i=1}^n   as_{\lambda} ( \varphi_i) as_{\lambda} ( \varphi_i^\circ) .
$$
The following Blaschke Santal\'o type inequality was proved in \cite{CFGLSW},  
\begin{equation}\label{asa-BS}
 as_{\lambda} ( \varphi) as_{\lambda} ( \varphi^\circ)  \leq (2 \pi)^n , 
\end{equation} 
where $\varphi$ is a log-concave function with barycenter at $0$. It was proved  in  \cite{CFGLSW} that equality  holds  if and only if $\varphi(x)=c e^{-\frac{1}{2} \langle A x, x \rangle}$ where $c>0$ is a constant and $A$ is a $(n \times n)$  positive definite matrix. Thus, the statement of the theorem follows. 
 By the duality formula (\ref{mixed-affine-polar-identity}) and  (\ref{affineGaussian}), it is easy to see that equality holds in (\ref{BS-mixed})  if $\varphi_i (x)= c_i e^{-\frac{1}{2} \langle A x, x \rangle}$.
On the other hand, if equality holds in (\ref{BS-mixed}), then equality holds in particular, for all $i$, in the inequality (\ref{asa-BS}) which means that 
for all $i$, $\varphi_i(x) =  c_i e^{-\frac{1}{2} \langle A_i x, x \rangle}$, where
 $c_i$ is a positive constant  and $A_i$ is a $(n \times n)$  positive definite  matrix. Note that for $\varphi_i(x) =  c_i e^{-\frac{1}{2} \langle A_i x, x \rangle}$, the dual function is  $\varphi_i^\circ (x) =  c_i^{-1} e^{-\frac{1}{2} \langle A_i^{-1} x, x \rangle}$. Thus, also using (\ref{affineGaussian}), the
equality condition leads to the following identity
 \begin{equation}\label{newidentity2}
\left( \det (A_1 + \cdots + A_n) \ \det (A_1^{-1} + \cdots + A_n^{-1} )  \right)^{\frac{1}{2}} = n^n
\end{equation} 
Therefore, by (\ref{BM+GA}), we must have for all 
$i$, $A_i= \lambda A$, for some $\lambda>0$ and for some positive definite matrix $A$.
Hence we have that $\varphi_i(x) = c_i e^{ - \langle A x , x \rangle / 2  }$.
\end{proof}

\vskip 2mm
The next proposition gives a monotonicity behavior of the  mixed $L_\lambda$-affine surface area. The proofs follow by  H\"older's inequality (see also \cite{CaglarWerner}).
\begin{prop} \label{theo:mono1}
Let  $\alpha \neq \beta,  \lambda \neq \beta$ be
real numbers. Let $\varphi_1, \cdots, \varphi_n: \R^{n}\rightarrow [0, \infty) $ be log concave functions.
\par
(i) If $1 \leq \frac{\alpha-\beta}{\lambda-\beta} < \infty$, then 
$
as_{\lambda } (\vec \varphi)\leq \big(as_{\alpha}  (\vec \varphi) \big)^{\frac{\lambda-\beta}{\alpha-\beta}}
\big(as_{\beta} (\vec \varphi)\big)^{\frac{\alpha-\lambda}{\alpha-\beta}}.
$
\par
(ii) If $1 \leq \frac{\alpha}{ \lambda} < \infty$, then 
$
as_{\lambda}(\vec \varphi) \leq \left(as_{\alpha}( \vec \varphi)\right)^\frac{\lambda}{\alpha} \big(\int (\varphi_1 \cdots \varphi_2)^{\frac{1}{n}} \big)^{\frac{\alpha-\lambda}{\alpha}}.
$
\par
(iii)  If $ \beta \leq \lambda $, then 
$ as_{\lambda }(\vec \varphi) \leq  \big(as_{ \infty  }(\vec \varphi) \big)^{ \lambda - \beta}
 \ as_{\beta  } ( \vec \varphi).
$ 
\par
\noindent
If $ \frac{\alpha-\beta}{\lambda-\beta}=1$ in (i),  respectively $\frac{\alpha}{\lambda}=1$ in (ii), then $\alpha=\lambda$ and equality holds trivially in (i)
respectively (ii). 
Equality also holds if for $1\leq i \leq n$, $\varphi_i(x)= c_i e^{- \frac{1}{2} \langle A_i x,x\rangle}$,  where
 $c_i$ is a positive constant  and $A_i$ is a $(n \times n)$  positive definite  matrix.
\end{prop}
\vskip 3mm

It follows from  Proposition \ref{theo:mono1} (ii) that for $0 < \lambda \leq \alpha$,
$$
0 \leq \left(\frac{as_{\lambda}(\vec \varphi) }{ \int (\varphi_1 \cdots \varphi_n )^{\frac{1}{n}} dx }\right)^\frac{1}{\lambda} \leq \left( \frac{as_{\alpha } (\vec \varphi) }{ \int (\varphi_1 \cdots \varphi_n )^{\frac{1}{n}} dx } \right)^{\frac{ 1}{ \alpha}},
$$
which means that  for $\lambda >0$ the function 
$\left(\frac{as_{\lambda }(\vec \varphi) }{ \int  \left(\varphi_1 \cdots \varphi_n \right)^{\frac{1}{n}}dx }    \right)^\frac{1}{\lambda} $  is bounded below by $0$ and
is  increasing for $ \lambda > 0$. 
Therefore, the  limit 
\begin{equation} \label{omega}
\Omega_{\vec \varphi } =  \lim _{ \lambda \downarrow 0}  \left(\frac{as_{\lambda }(\vec \varphi) }{ \int  \left(\varphi_1 \cdots \varphi_n \right)^{\frac{1}{n}} dx }\right)^\frac{1}{\lambda}
\end{equation} 
exists and the quantity $\Omega_{\vec \varphi }$ is  invariant under self adjoint $SL(n)$ maps. This quantity was first introduced by Paouris and Werner in \cite{PaourisWerner2011} for convex bodies, then 
by Caglar and Werner  \cite{CaglarWerner} for log concave functions using  $L_\lambda$-affine surface area.
It also follows from  Proposition \ref{theo:mono1} (ii) that for $\lambda <0$, the  function $ \lambda \rightarrow   \left(\frac{as_{\lambda }(\vec \varphi) }{ \int \left( \varphi_1 \cdots \varphi_n \right)^{\frac{1}{n}} dx }\right)^\frac{1}{\lambda} $  is  increasing. Therefore, $\lim _{ \lambda \uparrow 0}  \left(\frac{as_{\lambda }(\vec \varphi) }{ \int \left(\varphi_1 \cdots \varphi_n \right)^{\frac{1}{n}} dx }\right)^\frac{1}{\lambda} $ exists and,  in fact, is equal to $\Omega_{\vec \varphi } $.
\vskip 3mm

The quantity $\Omega_{\vec \varphi }$  is related to the relative entropy as follows.
\vskip 2mm
\begin{prop}\label{affineinvariant-prop}
Let $\varphi_i: \R^{n}\rightarrow [0, \infty) $ be log concave functions, $i=1, \cdots, n$. Then
$$ \Omega_{\vec \varphi } = \exp
\left[ \frac{ D_{KL} \left( P_{\prod_{i=1}^n  \varphi_i^{\frac{1}{n} } }   || Q _{\prod_{i=1}^n  \varphi_i^{\frac{1}{n}}}\right) }{\int \prod_{i=1}^n \varphi_i^\frac{1}{n} dx} + 
 \hskip -2mm {\displaystyle \int }   \log \left( 
 \frac{\prod_{i=1}^n \left( \mbox{det} \left[ \text{Hess} \left(-\log \varphi_i \right)\right]\right) ^{\frac{1}{n}}}
 {\mbox{det} \left[ \frac{1}{n} \sum_{i=1}^n \text{Hess} \left(-\log \varphi_i \right)\right]   }
\right) d \mu
 \right] \hskip -1mm,$$
where $d \mu = \frac{ \prod_{i=1}^n  \varphi_i^{\frac{1}{n}} dx}{ \int \prod_{i=1}^n  \varphi_i^{\frac{1}{n}} dx}$. 
\end{prop}
\vskip 2mm
\begin{proof}[Proof]
By definition and de l'H\^ospital, 
\begin{eqnarray*}
 \Omega_{\vec \varphi } &=&  \lim _{ \lambda \downarrow 0}  \left(\frac{as_{\lambda }(\vec \varphi) }{ \int  \left(\varphi_1 \cdots \varphi_n \right)^{\frac{1}{n}} dx }\right)^\frac{1}{\lambda}
 = \lim _{ \lambda \downarrow 0} \  \exp \left( \frac{1}{\lambda} 
   \log \left( \frac{as_{\lambda }(\vec \varphi) }{ \int  \left(\varphi_1 \cdots \varphi_n \right)^{\frac{1}{n}} dx } \right) 
 \right)  \\ 
 &=& \exp \left(  \lim _{ \lambda \downarrow 0}  \frac{ {\displaystyle \int }  \frac{d}{d\lambda} \prod_{i=1}^n    \bigg[ \varphi_i  \  \left(\frac{e^{\frac{\langle\grad\varphi_i , x \rangle}{\varphi_i}}}{\varphi_i^2}  \  \mbox{det} \left[  \text{Hess} \left(-\log \varphi_i \right)\right]  \right)^\lambda \bigg]^{\frac{1}{n}} dx}{as_\lambda(\vec \varphi)} 
 \right) \\
&=& \exp \left(   \frac{ {\displaystyle \int }  \prod_{i=1}^n  \varphi_i^{\frac{1}{n}} \  \log \left[ \prod_{i=1}^n   \left(\frac{e^{\frac{\langle\grad\varphi_i, x \rangle}{\varphi_i }}}{\varphi_i ^2}  \  \mbox{det} \left[  \text{Hess} \left(-\log \varphi_i \right)\right] \right)^{\frac{1}{n}}  \right] dx }{ \int  \left(\varphi_1 \cdots \varphi_n \right)^{\frac{1}{n}}   dx } 
 \right)  .
 \end{eqnarray*}
Now we treat the exponent further. As $q_{\prod_{i=1}^n  \varphi_i^{\frac{1}{n}}}= \prod_{i=1}^n  \varphi_i^{\frac{1}{n}} = \prod_{i=1}^n q_{ \varphi_i^{\frac{1}{n}}}$ and 

$$
p_{\prod_{i=1}^n  \varphi_i^{\frac{1}{n}}} = \prod_{i=1}^n   \frac{e^{\frac{1}{n} \frac{\langle\grad\varphi_i, x \rangle}{\varphi_i }}}{\varphi_i ^{\frac{1}{n}} }  \  
\mbox{det} \left[  \text{Hess} \left(-\log  \prod_{i=1}^n \varphi_i^{\frac{1}{n}} \right)\right]  , 
$$
we get that 
\begin{eqnarray*}
&& {\displaystyle \int }   \prod_{i=1}^n  \varphi_i^{\frac{1}{n}} \  \log \left[ \prod_{i=1}^n   \left(\frac{e^{\frac{\langle\grad\varphi_i, x \rangle}{\varphi_i }}}{\varphi_i ^2}  \  \mbox{det} \left[  \text{Hess} \left(-\log \varphi_i \right)\right] \right)^{\frac{1}{n}}  \right] dx   \\
&& = {\displaystyle \int}  q_{\prod_{i=1}^n  \varphi_i^{\frac{1}{n}}} \log\left( \frac{p_{\prod_{i=1}^n  \varphi_i^{\frac{1}{n}}}}{q_{\prod_{i=1}^n  \varphi_i^{\frac{1}{n}}}} 
\    \frac{\prod_{i=1}^n \left( \mbox{det} \left[  \text{Hess} \left(-\log \varphi_i \right)\right]\right) ^{\frac{1}{n}}}
 {\mbox{det} \left[  \text{Hess} \left(-\log  \prod_{i=1}^n \varphi_i ^\frac{1}{n}\right)\right] }
 \right) dx\\
 && = D_{KL} \left( P_{\prod_{i=1}^n  \varphi_i^{\frac{1}{n} } }   || Q _{\prod_{i=1}^n  \varphi_i^{\frac{1}{n}}}\right) + 
 {\displaystyle \int }  \prod_{i=1}^n  \varphi_i^{\frac{1}{n}} \  \log \left( 
 \frac{\prod_{i=1}^n \left( \mbox{det} \left[ \text{Hess} \left(-\log \varphi_i \right)\right]\right) ^{\frac{1}{n}}}
 {\mbox{det} \left[ \frac{1}{n} \sum_{i=1}^n \text{Hess} \left(-\log \varphi_i \right)\right]  }
\right).
\end{eqnarray*}
\end{proof}
\vskip 3mm
\begin{cor} 
Let $\varphi_i: \R^{n}\rightarrow [0, \infty) $ be log concave functions, $i=1, \cdots, n$. Then
$$
\log\  \left(\Omega_{\vec \varphi } \right) \leq  \frac{ D_{KL} \left( P_{\prod_{i=1}^n  \varphi_i^{\frac{1}{n} } }   || Q _{\prod_{i=1}^n  \varphi_i^{\frac{1}{n}}}\right) }{\int \prod_{i=1}^n \varphi_i^\frac{1}{n} dx}.
$$
If $n=1$, equality holds trivially. Otherwise, 
equality holds if  and only if one of  $\text{Hess} \left(-\log \varphi_i \right)$, $1 \leq i \leq n$,  is null or all 
are effectively proportional.
\end{cor}
\vskip 2mm
\begin{proof} For $i=1, \dots, n$, we put $H_i = \text{Hess} \left(-\log \varphi_i \right)$.  Then, by Proposition \ref{affineinvariant-prop}, 
\begin{eqnarray*}
\Omega_{\vec \varphi } = \exp 
\left[ \frac{ D_{KL} \left( P_{\prod_{i=1}^n  \varphi_i^{\frac{1}{n} } }   || Q _{\prod_{i=1}^n  \varphi_i^{\frac{1}{n}}}\right) }{\int \prod_{i=1}^n \varphi_i^\frac{1}{n} dx}\right]  \  \exp\left[ 
{\displaystyle \int}  \log \left( 
 \frac{\prod_{i=1}^n \left( \mbox{det}  H_i \right) ^{\frac{1}{n}}}
 {\mbox{det} \left[ \frac{1}{n} \sum_{i=1}^n H_i \right]   }
\right) d \mu
 \right] .
 \end{eqnarray*}
 It is easy to see that equality holds if $n=1$. Otherwise,  by (\ref{BM+GA}), 
\begin{equation*} 
\prod_{i=1}^{ n}  \left(  \det (H_i ) \right) ^\frac{1}{n}  \leq   \frac{1}{n^n} \   \left(  \det \left( \sum_{i=1}^n H_i \right) \right)=  \det \left( \frac{1}{n} \sum_{i=1}^n H_i \right) ,
\end{equation*}
with equality if and only if for all
$i$, $H_i= \lambda H$, for some $\lambda>0$ and  $H= \text{Hess} \left(-\log \varphi \right)$, for some log concave $\varphi$.
Therefore,
\begin{eqnarray*}
\Omega_{\vec \varphi } \leq  \exp 
\left[ \frac{ D_{KL} \left( P_{\prod_{i=1}^n  \varphi_i^{\frac{1}{n} } }   || Q _{\prod_{i=1}^n  \varphi_i^{\frac{1}{n}}}\right) }{\int \prod_{i=1}^n \varphi_i^\frac{1}{n} dx}\right] ,
\end{eqnarray*}
with equality if and only if for all $i$, $\text{Hess} \left(-\log \varphi_i \right)= \text{Hess} \left(-\log \varphi \right)$, for some log concave $\varphi$, i.e.,  $\text{Hess} \left(-\log \varphi_i \right)$
are all effectively proportional.

 \end{proof}
\vskip 3mm

\begin{cor} \label{affineinvariant-cor}
Let $ \varphi_i :\R^{n}\rightarrow [0, \infty) $ be  log concave functions,  $i=1, \cdots, n$. 
\par
(i) $\Omega_{ \vec \varphi} \leq   \left(\frac{as_{\lambda }(\vec \varphi) }{ \int \left(\varphi_1 \cdots \varphi_n \right)^{\frac{1}{n}} dx }\right)^\frac{1}{\lambda}$ for all $  \lambda > 0$ and  
$\Omega_{\vec \varphi}  \geq  \left(\frac{as_{\lambda }(\vec \varphi) }{ \int \left(\varphi_1 \cdots \varphi_n \right)^{\frac{1}{n}} dx }\right)^\frac{1}{\lambda}$ for all $  \lambda < 0 $.
\par
(ii) Let $\varphi_i =  a_i \varphi $ for some log concave function $ \varphi :\R^{n}\rightarrow [0, \infty)$ and $ a_i >0$. Then
$$ \Omega_{\vec \varphi}  \ \Omega_{\vec \varphi^\circ}  \ \leq  \ 1. $$
\par
(iii) Let $\varphi_i =  a_i \varphi $ for some log concave function $ \varphi :\R^{n}\rightarrow [0, \infty)$ and $ a_i >0$. Then
$$  \Omega_{ \vec \varphi} = \lim_{\alpha \rightarrow 1}  \left( \frac{as_\alpha( \vec \varphi^\circ ) }{ \int (\varphi_1 \cdots \varphi_n )^{\frac{1}{n}} dx } \right)^\frac{1}{ 1- \alpha} . $$
\par
\noindent
Equality holds in (i) and (ii)  if $ \varphi_i =  c_i e^{-\frac{1}{2}\langle Ax,x\rangle} $ where $c_i >0$, $1 \leq i \leq n$,   and $A$ is a $(n \times n)$  positive definite matrix.
\end{cor}
\vskip 2mm
\begin{proof}[Proof]
(i) is deduced immediately from the monotonicity behavior of the function $ \lambda \rightarrow   \left(\frac{as_{\lambda }(\vec \varphi) }{ \int \left(\varphi_1 \cdots \varphi_n \right)^{\frac{1}{n}} dx }\right)^\frac{1}{\lambda} $ and the definition of  $\Omega_{\vec \varphi}$.
\vskip 2mm
\noindent
(ii) By  (i) and  Proposition \ref{lambda-duality},   
$$
\Omega_{ \vec \varphi} \leq   \frac{as_{1}(\vec \varphi) }{ \int \left(\varphi_1 \cdots \varphi_n \right)^{\frac{1}{n}} dx }= \frac{as_{1}(\vec \varphi) }{ as_{0}(\vec \varphi) }
= \frac{as_{0}(\vec \varphi ^\circ) }{ as_{1}(\vec \varphi ^\circ) }, 
\hskip 4mm \Omega_{ \vec \varphi ^\circ} \leq \frac{as_{1}(\vec \varphi ^\circ) }{ as_{0}(\vec \varphi ^\circ) }. 
$$
\par
\noindent
(iii)  We use the  duality formula (\ref{mixed-affine-polar-identity}).
By definition
\begin{eqnarray*}
 \Omega_{ \vec \varphi^\circ} &=&  \lim _{ \lambda \rightarrow 0} \left(\frac{as_{\lambda }(\vec \varphi^\circ ) }{ \int \left(\varphi_1^\circ \cdots \varphi_n^\circ \right)^{\frac{1}{n}} dx }\right)^\frac{1}{\lambda} 
 = \lim _{ \lambda \rightarrow 0}  \left( \frac{as_{1-\lambda}( \vec \varphi) }{  \int  \left(\varphi_1^\circ \cdots \varphi_n^\circ \right)^{\frac{1}{n}}  dx  } \right)^\frac{1}{ \lambda}  \\
 &=& \lim _{ \alpha \rightarrow 1}  \left( \frac{as_{\alpha}( \vec \varphi) }{  \int \left(\varphi_1^\circ \cdots \varphi_n^\circ \right)^{\frac{1}{n}}  dx  }   \right)^\frac{1}{1- \alpha}.
\end{eqnarray*}
Therefore, $ \Omega_{\vec \varphi} = \lim_{\alpha \rightarrow 1}  \left( \frac{as_{\alpha}( \vec \varphi^\circ ) }{  \int  \left(\varphi_1 \cdots \varphi_n \right)^{\frac{1}{n}}  dx  }   \right)^\frac{1}{ 1- \alpha}$.
\end{proof}

\vskip 2mm
We define the {\em i-th mixed $L_\lambda$-affine surface area}  $as_{\lambda, i} (\vec \varphi)$  of $\vec \varphi  = ( \varphi_1, \varphi_2 )$ by
\begin{eqnarray*}\label{def:i-th-mixed-fdiv-logconc}
&& \hskip -7mm as_{\lambda, i} (\vec \varphi) = \\
&& \hskip -7mm {\displaystyle \int}  \hskip -1mm  \left[ \varphi_1   \left(
\frac{e^{\frac{\langle\grad\varphi_1, x \rangle}{\varphi_1}} }{\varphi_1^2}   \mbox{det} \left[  \text{Hess} \left( -\log \varphi_1 \right)\right] \right)^{\lambda} \right]^{\frac{i}{n}} \hskip-1mm
\left[ \varphi_2    \left(
\frac{e^{\frac{\langle\grad\varphi_2, x \rangle}{\varphi_2}} }{\varphi_2^2}   \mbox{det} \left[  \text{Hess} \left( -\log \varphi_2 \right)\right] \right)^{\lambda} \right]^{\frac{n-i}{n}} \hskip -5mm dx.
\end{eqnarray*}
Clearly, for all $\lambda$, $as_{\lambda, 0} (\vec \varphi) = as_{\lambda} ( \varphi_2) $ and $as_{\lambda, n } (\vec \varphi) = as_{\lambda} ( \varphi_1) $.
Moreover, $as_{0, n} (\vec \varphi) = \int_{} \varphi_1 dx $  and $as_{1, n} (\vec \varphi) = \int_{} \varphi_1^\circ dx $ (see, \cite{CaglarWerner}).
We also give a definition for  $ as_{\infty, i } (\vec \varphi)$ and  $ as_{-\infty, i} (\vec \varphi)$.
\begin{equation*}\label{as+infinity}
\hskip -1mm as_{\infty, i} (\vec \varphi) \ =   \  \max_{ x }   \left[ \frac{e^{\frac{\langle\grad\varphi_1, x \rangle}{\varphi_1}}}{\varphi_1^2}  \  \mbox{det} \left[  \text{Hess} \left(-\log \varphi_1 \right)\right] \right]^{\frac{i}{n}}  \left[ \frac{e^{\frac{\langle\grad\varphi_2, x \rangle}{\varphi_2}}}{\varphi_2^2}  \  \mbox{det} \left[  \text{Hess} \left(-\log \varphi_2 \right)\right] \right]^{\frac{n-i}{n}} \hskip -5mm.        
\end{equation*}
$$
as_{-\infty, i} (\vec  \varphi) =\frac{1}{ as_{\infty, i} (\vec \varphi)}.
$$
It is  easy to see that these expressions are  invariant under symmetric linear transformations with determinant $1$.
\vskip 3mm
\noindent
\textbf{Remarks.}
(i) It follows from Proposition \ref{lambda-duality} that $ as_{1-\lambda, i} (\vec \varphi) = as_{\lambda, i} (\vec {\varphi^\circ})$ where $\vec \varphi = ( \varphi_1, \varphi_2)$ such that $ \varphi_1 = a \varphi_2$, $a>0.$
\par
\noindent
(ii) Let $\varphi_l(x)= c_l e^{- \frac{1}{2} \langle A_l x,x\rangle}$,  where
 $c_l$ is a positive constant  and $A_l$ is a $(n \times n)$  positive definite  matrix for $ l=1,2$. Then,
\begin{equation}\label{affineGaussian3}
as_{\lambda, i} (\vec \varphi) = \  \left( c_1^{i} \ c_2^{n-i}  \right)^{\frac{1 - 2 \lambda}{n}} \  \left( (\det(A_1))^{i}  ( \det(A_2) )^{n-i} \right)^{\frac{\lambda}{n}} \ \frac{( 2 n \pi)^{\frac{n}{2}}}{ \left( \det( i A_1 + (n-i) A_2) \right)^{\frac{1}{2}}}.
\end{equation}
\noindent

\vskip 2mm

The next proposition is identical to Proposition \ref{theo:mono1} and the proof  follows by  H\"older's inequality.
\vskip 2mm
\begin{prop} \label{theo:mono2}
Let $i \in \R$ and  $\alpha \neq \beta,  \lambda \neq \beta$ be
real numbers. Let $\varphi_1, \varphi_2: \R^{n}\rightarrow [0, \infty) $ be log concave functions.
\par
(i) If $1 \leq \frac{\alpha-\beta}{\lambda-\beta} < \infty$, then 
$
as_{\lambda, i} (\vec \varphi)\leq \big(as_{\alpha, i}  (\vec \varphi) \big)^{\frac{\lambda-\beta}{\alpha-\beta}}
\big(as_{\beta, i} (\vec \varphi)\big)^{\frac{\alpha-\lambda}{\alpha-\beta}}.
$
\par
(ii) If $1 \leq \frac{\alpha}{ \lambda} < \infty$, then 
$
as_{\lambda, i}(\vec \varphi) \leq \left(as_{\alpha, i}( \vec \varphi)\right)^\frac{\lambda}{\alpha} \big(\int \varphi_1^{\frac{i}{n}} \varphi_2^{\frac{n-i}{n}}\big)^{\frac{\alpha-\lambda}{\alpha}}.
$
\par
(iii)  If $ \beta \leq \lambda $, then 
$ as_{\lambda,i}(\vec \varphi) \leq  \big(as_{ \infty, i }(\vec \varphi) \big)^{ \lambda - \beta}
 \ as_{\beta,i} ( \vec \varphi).
$ 
\par
\noindent
If $ \frac{\alpha-\beta}{\lambda-\beta}=1$ in (i),  respectively $\frac{\alpha}{\lambda}=1$ in (ii), then $\alpha=\lambda$ and equality holds trivially in (i)
respectively (ii). 
Equality also holds if  $\varphi_l(x)= c_l e^{- \frac{1}{2} \langle A_l x,x\rangle}$,  where
 $c_l$ is a positive constant  and $A_l$ is a $(n \times n)$  positive definite  matrix for $ l=1,2$.
\end{prop}

\vskip 3mm

The following proposition is a direct consequence of  Proposition \ref{i-th-mixed-prop}.
\begin{prop}\label{i-th-mixed-prop-logc}
Let $ \varphi_1, \varphi_2:\R^{n}\rightarrow [0, \infty)$ be  log concave functions.
If $j \leq i \leq k$ or $k \leq i \leq j$, then
\begin{eqnarray*}\label{def:i-th-mixed-fdiv-log}
as_{\lambda, i} (\vec \varphi)   \  \leq  \  \left[ as_{\lambda, j } (\vec \varphi)  \right]^{\frac{k-i}{k-j}} \times  \  \left[ as_{\lambda, k} (\vec \varphi)  \right]^{\frac{i-j}{k-j}} .
\end{eqnarray*}
Equality holds trivially if $ i=k$ or $i=j$. 
Otherwise, equality holds if and only if  one of the functions $\varphi_l  \left(
\frac{e^{\frac{\langle\grad\varphi_l, x \rangle}{\varphi_l}} }{\varphi_l^2}   \mbox{det} \left[  \text{Hess} \left( -\log \varphi_l \right)\right] \right)^{\lambda} $, $ l=1,2$,  
is null or they are effectively proportional.
\end{prop}
\vskip 2mm
\noindent
In Proposition \ref{i-th-mixed-prop-logc}, if we let $j =0$ and $k=n$, then for all $\lambda$ and $0 \leq i \leq n$
\begin{eqnarray}\label{def:i-th-mixed-fdiv-log2}
\left[ as_{\lambda, i} (\vec \varphi)  \right]^n  \  \leq  \  \left[ as_{\lambda } ( \varphi_2)  \right]^{n-i}   \  \left[ as_{\lambda} (\varphi_1)  \right]^{i} .
\end{eqnarray}
If we let $i =0$ and $j=n$, then for all $\lambda$ and $ k \leq 0$
\begin{eqnarray}\label{def:i-th-mixed-fdiv-log3}
\left[ as_{\lambda, k} (\vec \varphi)  \right]^n  \  \geq  \  \left[ as_{\lambda } ( \varphi_2)  \right]^{n-k}   \  \left[ as_{\lambda} (\varphi_1)  \right]^{k} .
\end{eqnarray}
From inequality (\ref{def:i-th-mixed-fdiv-log2}) and an inequality  of \cite{CFGLSW}, already quoted here as inequality (\ref{asa-BS}), one gets for functions with barycenter at $0$,
\begin{eqnarray*}\label{def:i-th-mixed-fdiv-log4}
 \left[as_{\lambda, i} (\varphi_1, \varphi_2 )  \right]^n  \left[as_{\lambda, i} (\varphi_1^\circ, \varphi_2^\circ)\right]^n   &&  \leq  \   \left[ as_{\lambda } ( \varphi_2) as_{\lambda } ( \varphi_2^\circ) \right]^{n-i}   \  \left[ as_{\lambda} (\varphi_1)  as_{\lambda} (\varphi_1^\circ) \right]^{i} \nonumber   \\
&&  \ 
\leq (2 \pi)^{n^2 } 
\end{eqnarray*}
holds true for all $\lambda \in [0,1]$ and $0 \leq i \leq n$. 
Hence, we have proved the following proposition which also follows directly from Proposition \ref{duality-mixed-log}.
\vskip 2mm
\begin{prop}
Let $\varphi_1, \varphi_2$ be log concave functions with barycenter at $0$. 
If $\lambda \in [0,1]$ and $0 \leq i \leq n$, then
\begin{eqnarray}\label{def:i-th-mixed-fdiv-log4}
 as_{\lambda, i} (\vec \varphi)  as_{\lambda, i} (\vec {\varphi^\circ} )  \ 
 \leq (2 \pi)^{n } .
\end{eqnarray}
Equality holds if and only if $\varphi_l = c_l e^{ -\frac{1}{2}  \langle Ax, x \rangle }$ where $c_l >0$, $ l =1,2$, and $A$ is a $(n \times n)$  positive definite matrix.
\end{prop}
\begin{proof}
The inequality follows from above. Using (\ref{affineGaussian3}) and the duality formula  $ as_{1-\lambda, i} (\vec \varphi) = as_{\lambda, i} (\vec {\varphi^\circ})$, it is easy to see that equality holds  in (\ref{def:i-th-mixed-fdiv-log4}) if $\varphi_l = c_l e^{ -\frac{1}{2}  \langle Ax, x \rangle }$ where $c_l >0$ and $A$ is a $(n \times n)$  positive definite matrix. On the other hand, if equality holds in (\ref{def:i-th-mixed-fdiv-log4})   
 then equality holds in particular, for $ l =1,2$, in the inequality (\ref{asa-BS}) which means that, $\varphi_l(x) =  c_l e^{-\frac{1}{2} \langle A_l x, x \rangle}$, where
 $c_l$ is a positive constant  and $A_l$ is a $(n \times n)$  positive definite  matrix.
Note that for $\varphi_l(x) =  c_l e^{-\frac{1}{2} \langle A_l x, x \rangle}$, the dual function is  $\varphi_l^\circ (x) =  c_l^{-1} e^{-\frac{1}{2} \langle A_l^{-1} x, x \rangle}$. Thus, also using (\ref{affineGaussian3}), the
equality condition leads to the following identity
 \begin{equation}\label{newidentity3}
\left( \det (i A_1  + (n-i)A_2) \ \det (i A_1^{-1} + (n-i)A_2^{-1} )  \right)^{\frac{1}{2}} = n^n
\end{equation} 
Therefore, by (\ref{BM+GA}), we must have $A_1 = A_2$.
Hence we have that $\varphi_l(x) = c_l e^{ - \langle A x , x \rangle / 2  }$.
\end{proof}

\vskip 2mm 
\noindent 
Umut Caglar\\
{\small Department of Mathematics} \\
{\small Case Western Reserve University} \\
{\small Cleveland, Ohio 44106, U. S. A.} \\
{\small \tt umut.caglar@case.edu}\\ \\
\noindent
\and 
Elisabeth Werner\\
{\small Department of Mathematics \ \ \ \ \ \ \ \ \ \ \ \ \ \ \ \ \ \ \ Universit\'{e} de Lille 1}\\
{\small Case Western Reserve University \ \ \ \ \ \ \ \ \ \ \ \ \ UFR de Math\'{e}matique }\\
{\small Cleveland, Ohio 44106, U. S. A. \ \ \ \ \ \ \ \ \ \ \ \ \ \ \ 59655 Villeneuve d'Ascq, France}\\
{\small \tt elisabeth.werner@case.edu}\\ \\

\end{document}